\documentclass[11pt]{amsart}

\pdfoutput=1

\usepackage{etex,amsfonts, amsmath, amsthm, amssymb, epsfig, graphics,
  psfrag, latexsym, mathtools, mathrsfs,enumitem, subfig,
  longtable, booktabs, yfonts, centernot,ifthen,eso-pic,longtable,stmaryrd}
\newcounter{dummycounter}
\usepackage[normalem]{ulem}
\usepackage{xr-hyper}
\usepackage[bbgreekl]{mathbbol}
\DeclareSymbolFontAlphabet{\mathbb}{AMSb}
\DeclareSymbolFontAlphabet{\mathbbl}{bbold}
\usepackage{xcolor}
\usepackage[all,cmtip]{xy}
\usepackage{tikz} \usetikzlibrary{matrix,arrows,shapes,calc,arrows.meta,positioning}
\pgfdeclarelayer{background}
\pgfdeclarelayer{foreground}
\pgfsetlayers{background,main,foreground}
\tikzset{zlevel/.style={%
    execute at begin scope={\pgfonlayer{#1}},
    execute at end scope={\endpgfonlayer}
  }}
\tikzset{khdiff/.style={%
    black!30,opacity=1,->
  }}
\tikzset{knot/.style={%
    black,thick,preaction={draw,white,line width=3pt}
  }}
\tikzset{resol/.style={%
    black,thick
  }}
\tikzset{linelabel/.style={%
    outer sep=0,inner sep=1pt,fill=white,fill opacity=1,draw opacity=1
  }}
\tikzset{preservemap/.style={%
    ->,black,opacity=1
  }}
\tikzset{dropmap/.style={%
    ->,black!30,opacity=1
  }}
\tikzset{tightnode/.style={%
    outer sep=0,inner sep=1pt
  }}
\tikzset{tighternode/.style={%
    outer sep=0,inner sep=0pt
  }}
\tikzset{bfill/.style={%
    fill=white,fill opacity=0.5,text opacity=1
  }}

\definecolor{darkblue}{rgb}{0,0,0.4} 
 \usepackage[colorlinks=true, citecolor=darkblue, filecolor=darkblue, linkcolor=darkblue,urlcolor=darkblue]{hyperref} 
\usepackage[all]{hypcap}
\usepackage{xr-hyper}

\graphicspath{{draws/}{}}

\numberwithin{equation}{section}

\newtheorem{lem}{Lemma}[section]               

\newtheorem{theorem}[lem]{Theorem}
\newtheorem{lemma}[lem]{Lemma}

\newtheorem{porism}[lem]{Porism}
\newtheorem{corollary}[lem]{Corollary}               

\newtheorem{proposition}[lem]{Proposition}

\theoremstyle{definition}

\theoremstyle{remark}

\newtheorem{remark}[lem]{Remark}

\numberwithin{figure}{section}
\numberwithin{table}{section}

\newcommand{\R}{\mathbb{R}}

\newcommand{\FF}{\mathbb{F}}

\newcommand{\mf}{\mathfrak}

\newcommand{\wt}{\widetilde}
\newcommand{\ol}{\overline}

\newcommand{\from}{\colon}
\newcommand{\into}{\hookrightarrow}

\newcommand{\lra}{\longrightarrow}

\renewcommand{\th}{^{\text{th}}}

\newcommand{\SpinC}{\text{Spin}^{\text{C}}}

\newcommand{\HF}{\mathit{HF}}

\newcommand{\HFK}{\mathit{HFK}}

\newcommand{\HFa}{\widehat{\HF}}
\newcommand{\HFKa}{\widehat{\HFK}}

\DeclareMathOperator{\Id}{Id}

\newcommand{\Kh}{\mathit{Kh}}

\newcommand{\KhCx}{\mathit{CKh}}

\newcommand{\AKh}{\mathit{AKh}}
\newcommand{\AKhCx}{\mathit{ACKh}}
\newcommand{\AV}{\mathit{AV}}

\newcommand{\co}{\colon}
\newcommand{\bdy}{\partial}
\newcommand{\RR}{\R}

\DeclareMathOperator{\rank}{rank}

\newcommand{\ZZ}{\mathbb{Z}}

\DeclareMathOperator{\Cone}{Cone}

\newcommand{\BurnsideCat}{\mathscr{B}}

\newcommand{\AbelianGroups}{\mathsf{Ab}}

\newcommand{\Spaces}{\mathsf{Top}}

\newcommand{\KhSpace}{\mathscr{X}}
\newcommand{\AKhSpace}{A\KhSpace}
\newcommand{\AF}{\mathit{AF}}

\newcommand{\spinc}{\mathfrak{s}}

\begin{document}


\title{Khovanov homology of strongly invertible knots and their quotients}

\author{Robert Lipshitz}
\thanks{\texttt{RL was supported by NSF Grant DMS-1810893.}}
\email{\href{mailto:lipshitz@uoregon.edu}{lipshitz@uoregon.edu}}
\address{Department of Mathematics, University of Oregon, Eugene, OR 97403}

\author{Sucharit Sarkar}
\thanks{\texttt{SS was supported by NSF Grant DMS-1905717.}}
\email{\href{mailto:sucharit@math.ucla.edu}{sucharit@math.ucla.edu}}
\address{Department of Mathematics, University of California, Los Angeles, CA 90095}


\keywords{}

\date{\today}

\begin{abstract}
  We construct a spectral sequence relating the Khovanov homology of a
  strongly invertible knot to the annular Khovanov homologies of the
  two natural quotient knots. Using this spectral sequence, we
  re-prove that Khovanov homology distinguishes certain slice disks.
  We also give an analogous spectral sequence for $\HFa$ of the
  branched double cover.
\end{abstract}
\maketitle
\vspace{-1cm}


\tableofcontents


\section{Introduction}\label{sec:intro}

A knot $K\subset \RR^3$ is called \emph{strongly invertible} if $K$
intersects some straight line in exactly two points, and is preserved
setwise by rotation by $180^\circ$
around that line; this straight line is called the \emph{axis}.  While
strongly invertible knots have been studied for decades
(see~\cite{Sakuma-top-invert} and the references therein), they have
recently seen a surge in interest. For example, there is a
somewhat mysterious concordance group of strongly invertible
knots~\cite{Sakuma-top-invert}, as well as 
natural equivariant analogues of the slice
genus~\cite{BI-top-invert-g4}. Related to this, many of the pairs of
non-isotopic slice disks or more general slice surfaces which have
appeared in the literature recently come from strongly invertible
knots~\cite{Hay-top-corks,SS-kh-surf,HS-kh-exotic}, a phenomenon which
has led to connections with Heegaard Floer-theoretic
invariants~\cite{DMS-hf-invert}. In addition to Heegaard Floer
homology, Donaldson's diagonalization theorem~\cite{BI-top-invert-g4},
the $G$-signature theorem~\cite{AB-top-invert-AS}, and Khovanov
homology~\cite{Couture-kh-divides,Watson-kh-invert,Snape-kh-inversion,LW-kh-invert} have also been applied
recently to study strongly invertible knots.

The quotient of $K$ by its strong inversion is naturally an embedded arc with
boundary on the axis. By gluing this arc to part of the axis, we
obtain a quotient knot; see the first two pictures of
Figure~\ref{fig:9-46-quotients}. (In $S^3$, instead of $\RR^3$, there
are two equally natural choices of quotient knot.) The main goal of
this paper is to construct a spectral sequence relating the Khovanov
homology of a strongly invertible knot $K$ and a variant of the
Khovanov homology of its quotient.

Similar results have been proved before, for other kinds of
symmetries. Stoffregen-Zhang~\cite{SZ-kh-localization} and
Borodzik-Politarczyk-Silvero~\cite{BPS-kh-periodic} showed that there
is a spectral sequence relating the Khovanov homology of a periodic
knot (a knot preserved by rotation around an axis disjoint from it)
and the annular Khovanov homology of its quotient (see
also~\cite{Cornish-kh-loc,Zhang18:periodic}). An analogous result relating
the symplectic Khovanov homology of a 2-periodic knot and of its
quotient was proved earlier by Seidel-Smith~\cite{SS-kh-loc}. Using
the same technical tool, Hendricks proved a similar relationship for
knot Floer homology~\cite{Hen-hf-periodic}, and analogous results have
been given for other symmetries in Heegaard Floer
homology~\cite{Hen-hf-dcov,HLS-flexible,LT-hf-HH,LT-hf-covering,
  Large-hf-loc,HLL-hf-dcov}.

Like Stoffregen-Zhang's and Borodzik-Politarczyk-Silvero's spectral
sequence for periodic knots, the spectral sequence we construct for
strongly invertible knots relates the Khovanov homology of $K$ to the
annular Khovanov homology of its quotient $\ol{K}$. If $\tau$ denotes
the $180^\circ$ rotation around the axis then by a slight
$\tau$-equivariant perturbation of the knot $K$ we may assume that the
projection of $K$ to the plane perpendicular to the axis is a knot
diagram. Following the
literature~\cite{Boyle-top-alt-inv,BI-top-invert-g4}, we call such a
diagram \emph{intravergent}.  The quotient knot $\ol{K}$ may be viewed
as an annular knot in two natural ways, $\ol{K}_0$ and $\ol{K}_1$,
corresponding to taking the quotients of the $0$-resolution or the
$1$-resolution of the fixed crossing of the intravergent diagram $K$;
see Figure~\ref{fig:9-46-quotients}. (These quotients depend on the
diagram $K$; see Remark~\ref{rem:gr-shift} and
Proposition~\ref{prop:classical-invt}.)  The annular Khovanov chain
complexes of these knots are related by an \emph{axis-moving map}
$f^+\from\Sigma^{0,0,1}\AKhCx(\ol{K}_1)\to\AKhCx(\ol{K}_0)$, where
$\Sigma^{a,b,c}$ denotes a (homological, quantum, annular) trigrading
shift by $(a,b,c)$; we introduce the map $f^+$, which is a special case of
Akhmechet-Khovanov's maps associated to anchored
cobordisms~\cite{AK-kh-anchored}, in
Section~\ref{sec:moving}. By a slight abuse of notation, we
define the annular Khovanov chain complex of the pair of annular knots
$(\ol{K}_1,\ol{K}_0)$ to be the mapping cone of $f^+$,
\[
  \AKhCx(\ol{K}_1,\ol{K}_0)=\Cone\big(
  \Sigma^{0,0,1}\AKhCx(\ol{K}_1)\stackrel{f^+}{\to} \AKhCx(\ol{K}_0)
  \big),
\] 
and the annular Khovanov homology of the pair to be the homology of this complex which, over any
field $\FF$, is also (unnaturally) isomorphic to the
homology of the mapping cone of the induced map on homology,
\begin{align*}
  \AKh(\ol{K}_1,\ol{K}_0;\FF)&=H_*\Cone\big(
  \Sigma^{0,0,1}\AKhCx(\ol{K}_1;\FF)\stackrel{f^+}{\to} \AKhCx(\ol{K}_0;\FF)
  \big)\\&\cong H_*\Cone\big(
  \Sigma^{0,0,1}\AKh(\ol{K}_1;\FF)\stackrel{f^+}{\to}
  \AKh(\ol{K}_0;\FF) \big).
\end{align*}

\begin{theorem}\label{thm:main}
  Given a strongly invertible knot $K$ with annular quotients
  $\ol{K}_0$, $\ol{K}_1$ there is a spectral sequence with the
  following properties:
  \begin{enumerate}[label=($\Theta$-\arabic*),leftmargin=*]
  \item The $E^1$-page is
    $\Kh(K;\FF_2)\otimes\FF_2[\theta^{-1},\theta]$ with
    $d^1$-differential the map $\theta(\Id+\tau_*)$,
    \[
      \begin{tikzpicture}[xscale=3.5]
        \foreach\i in {-1,0,1}{
          \node (t\i) at (\i,0) {$\theta^{\i}\Kh(K;\FF_2)$};}
        \node (t-2) at (-1.8,0) {$\cdots$};
        \node (t2) at (1.8,0) {$\cdots$,};
        
        \foreach \i [count=\j from -1] in {-2,-1,0,1}{
          \draw[->] (t\i)--(t\j) node[midway,anchor=south] {\small $\theta(\Id+\tau_*)$};
        }
      \end{tikzpicture}
    \]
    where $\tau_*$ is induced by the strong inversion.
  \item\label{item:ss-gr} The $d^r$-differential preserves the quantum grading and
    increases the $\theta$-power grading by $r$.
  \item\label{item:main-thm-e-infty} The spectral sequence converges to
    $\AKh(\ol{K}_1,\ol{K}_0;\FF_2)\otimes\FF_2[\theta^{-1},\theta]$.  Keeping
    track of quantum gradings, the summand of the spectral sequence in
    quantum grading $j$ converges to
    \begin{equation}
      \bigoplus_{\substack{\ol{\imath},\ol{\jmath},k\in\ZZ\\2\ol{\jmath}+k=j-1+3N_--6\ol{N}_-}}\hspace{-1em}H_*\Cone\big(
      \AKh_{\ol{\imath},\ol{\jmath},k-1}(\ol{K}_1;\FF_2)\stackrel{f^+}{\to}
      \AKh_{\ol{\imath},\ol{\jmath},k}(\ol{K}_0;\FF_2)
      \big)\otimes\FF_2[\theta^{-1},\theta]
    \end{equation}
    where $N_-$ (respectively $\ol{N}_-$) is the number of negative
    crossings of $K$ (respectively $\ol{K}$).
  \end{enumerate}
\end{theorem}
Since the quantity $3N_--6\ol{N}_-$ comes up frequently, let
\begin{equation}\label{eq:Delta}
\Delta=N_--2\ol{N}_-.
\end{equation}
See Remark~\ref{rem:gr-shift}
and~Proposition~\ref{prop:classical-invt} for a little further
discussion of the grading shift, and Section~\ref{sec:conventions} for
our grading conventions for Khovanov homology. Like the periodic knot
case, the proof of Theorem~\ref{thm:main} uses the Khovanov stable
homotopy type and Smith theory.

\begin{corollary}
  For any quantum grading $j$, we have
  \[
    \sum_i\dim \Kh_{i,j}(K;\FF_2) \geq
    \hspace{-2em}\sum_{\substack{\ol{\imath},\ol{\jmath},k\in\ZZ\\2\ol{\jmath}+k=j-1+3\Delta}}\hspace{-2em}\dim\AKh_{\ol{\imath},\ol{\jmath},k}(\ol{K}_1,\ol{K}_0;\FF_2).
    \]
\end{corollary}

\begin{proof}
  This follows from Theorem~\ref{thm:main} by comparing the ranks of the $E^1$-page and $E^\infty$-page.
\end{proof}

\begin{corollary}
  \label{cor:comp-tau-SS}
  Assume that in some quantum grading $j$, $\Kh_{*,j}(K;\FF_2)$ is
  supported in a single homological grading $i$. Then,
  \[
    \dim \Kh_{i,j}(K;\FF_2)-2\rank((\Id+\tau_*)_{i,j})=\hspace{-2em}\sum_{\substack{\ol{\imath},\ol{\jmath},k\in\ZZ\\2\ol{\jmath}+k=j-1+3\Delta}}\hspace{-2em}\dim\AKh_{\ol{\imath},\ol{\jmath},k}(\ol{K}_1,\ol{K}_0;\FF_2).
  \]
  where $(\Id+\tau_*)_{i,j}$ is the induced endomorphism on
  $\Kh_{i,j}(K;\FF_2)$.
\end{corollary}
\begin{proof}
  This also follows from Theorem~\ref{thm:main} by equating the
  ranks of the $E^2$-page and $E^\infty$-page: since $\theta$ has
  homological grading $(-1,0)$, statement~\ref{item:ss-gr} in the
  theorem implies that the $d^r$-differentials vanish for $r>1$.
\end{proof}

\begin{corollary}
  \label{cor:comp-tau}
  Assume in some quantum grading $j$, $\Kh_{*,j}(K;\FF_2)$ is
  supported in a single homological grading $i$.  Suppose also that
  $\bigoplus_{\ol{\imath},\ol{\jmath},k\mid2\ol{\jmath}+k=j-1+3\Delta}
  \AKh_{\ol{\imath},\ol{\jmath},k}(\ol{K}_1,\ol{K}_0;\FF_2)=0$. Then,
  the endomorphism $(\Id+\tau_*)_{i,j}$ on $\Kh_{i,j}(K;\FF_2)$ has
  rank $\tfrac{1}{2}\dim\Kh_{i,j}(K;\FF_2)$. In particular,
  if $\dim\Kh_{i,j}(K;\FF_2)=2$ then, up to a change of basis,
  $\tau_*$ is given by the matrix $ \left(\begin{smallmatrix}
      0 & 1\\
      1 & 0
  \end{smallmatrix}\right).
  $
\end{corollary}
\begin{proof}
  The first statement is immediate from
  Corollary~\ref{cor:comp-tau-SS}.  The second follows from the first
  and the fact that any involution of $\FF_2^2$ is one of
  $\left(\begin{smallmatrix}
      1 & 0\\
      0 & 1
  \end{smallmatrix}\right)$, $\left(\begin{smallmatrix}
      0 & 1\\
      1 & 0
  \end{smallmatrix}\right)$, $\left(\begin{smallmatrix}
      1 & 1\\
      0 & 1
  \end{smallmatrix}\right)$, or $\left(\begin{smallmatrix}
      1 & 0\\
      1 & 1
  \end{smallmatrix}\right)$, and the last two are conjugate to the second one.
\end{proof}

As noted above, one reason strongly invertible knots have appeared
recently is that they have furnished examples of non-isotopic pairs
of slice disks. It turns out that Corollary~\ref{cor:comp-tau} and
properties of the maps on Khovanov homology can be used to prove that certain pairs of slice
disks are distinguished by Khovanov homology, without explicitly
computing the maps associated to the slice disks. We illustrate this
phenomenon for the knot $9_{46}$ in Section~\ref{sec:slice}.

Another spectral sequence was constructed by
Lobb-Watson~\cite{LW-kh-invert}, although they used a different kind
of diagrams, transvergent rather than intravergent. (They also mention
intravergent diagrams briefly, in the discussion around their Figure
7.) It might be interesting to compare their $\mathcal{F}$ spectral
sequence with the one constructed here; in particular, this might give
an approach to proving that their other, $\mathcal{G}$, spectral
sequence collapses~\cite[Question 6.5]{LW-kh-invert}.

The reduced Khovanov homology of $K$ is closely related to the Heegaard Floer
homology of the branched double cover $\Sigma(S^3,K)$ of
$K$~\cite{OSz-hf-branched,Roberts-kh-dcov,GW-kh-sutured}, and
Theorem~\ref{thm:main} has an analogue for Heegaard Floer homology:
\begin{theorem}\label{thm:dcov}
  Given a strongly invertible knot $K$ with quotient knot
  $\ol{K}$ there is an ungraded spectral sequence with $E^1$-page
  given by $\HFa(\Sigma(S^3,K))\otimes\FF_2[\theta^{-1},\theta]]$
  converging to
  $\HFa(\Sigma(S^3,\ol{K}))\otimes\FF_2[\theta^{-1},\theta]]$.
\end{theorem}
As we will see, this follows from a localization result of Hendricks,
Lidman, and the first author for the Heegaard Floer homology of double
branched covers~\cite{HLL-hf-dcov}, which in turn follows from a
general localization theorem in Lagrangian intersection Floer homology
of Large~\cite{Large-hf-loc}. Note that in Theorem~\ref{thm:dcov},
there are two choices for the quotient knot $\ol{K}$, depending on
which half of the axis one chooses. The statement holds for either
choice. There is also an analogue for the knot Floer homology relative
to a preimage of the axis, Theorem~\ref{thm:HFK-dcov}.

We expect that the spectral sequence in Theorem~\ref{thm:main} is an
invariant of the strong inversion on $K$, but do not pursue this
here. (Invariance of the spectral sequence from
Theorem~\ref{thm:dcov} follows from
\cite[Remark~4.13]{HLL-hf-dcov}.) Theorem~\ref{thm:dcov} suggests
there might be an interesting reduced version of
Theorem~\ref{thm:main}, but we do not pursue that either. One could
also consider links in Theorems~\ref{thm:main} and~\ref{thm:dcov},
intersecting the axis in more than two points, but we also do not
pursue that generalization.

This paper is organized as follows. Since there are many conventions
for Khovanov homology, we review ours in
Section~\ref{sec:conventions}. We then introduce the axis-moving maps
on annular Khovanov homology and some of their basic properties, in
Section~\ref{sec:moving}. Section~\ref{sec:proof} proves the main
localization result, Theorem~\ref{thm:main}. We
give the application to slice disks in Section~\ref{sec:slice}. We end
with the proof of the analogue for Heegaard Floer homology,
Theorem~\ref{thm:dcov}, in Section~\ref{sec:HF}.

\subsection*{Acknowledgments.} We thank Champ Davis for helpful conversations
and computer code and Keegan Boyle, Mikhail Khovanov, Tye Lidman, Andrew Lobb, Matthew
Stoffregen, and Liam Watson for further helpful discussions.

\section{Conventions for Khovanov homology}\label{sec:conventions}

We start by describing our grading conventions for Khovanov homology
(which also serves as a terse review of Khovanov homology
itself). Given a knot diagram with $N$ crossings numbered $1$ through
$N$, consider the Kauffman cube of resolutions, where the complete
resolution of the diagram at the vertex
$v=(v_1,\dots,v_N)\in\{0,1\}^N$ is obtained by resolving the $i\th$
crossing
$\vcenter{\hbox{\begin{tikzpicture}[scale=0.4]
      \draw[knot](0,0)--(1,1);\draw[knot](1,0)--(0,1);\end{tikzpicture}}}$
by the $0$-resolution
$\vcenter{\hbox{\begin{tikzpicture}[scale=0.4]
      \draw[resol](0,0)to[out=45,in=135](1,0);\draw[resol](1,1)to[out=-135,in=-45](0,1);\end{tikzpicture}}}$
if $v_i=0$ or by the $1$-resolution
$\vcenter{\hbox{\begin{tikzpicture}[scale=0.4]
      \draw[resol](0,0)to[out=45,in=-45](0,1);\draw[resol](1,1)to[out=-135,in=135](1,0);\end{tikzpicture}}}$
if $v_i=1$, for each $i\in\{1,\dots,N\}$.  We will view the cube as
$(1\to 0)^N$, so each edge runs from a vector with a $1$ in
some coordinate to the corresponding vector with a $0$ in that
coordinate. Associated to each edge is an elementary saddle cobordism
between the corresponding resolutions.

The Frobenius algebra $\ZZ[X]/(X^2)$ (with comultiplication given by
$1\mapsto1\otimes X+X\otimes 1,X\mapsto X\otimes X$) corresponds to a
$2$-dimensional TQFT.  The Khovanov chain complex $\KhCx$ is obtained
by applying this TQFT to the cube of resolutions, and then taking the
total complex.  More concretely, $\KhCx$ is freely generated by the
Khovanov generators $x$, which consist of a choice of a vertex $v\in\{0,1\}^N$
and a labeling of the circles in the complete resolution at $v$
by the labels $\{1,X\}$. The homological grading of $x$ is
$|v|-N_-$, and the quantum grading is
$N-3N_-+|v|+\#\{\text{circles labeled }X\}-\#\{\text{circles labeled
}1\}$. (Here $N_-$ is the number of negative crossings in the diagram
and $|v|=\sum_i v_i$ is the $L^1$-norm of $v$.) So, the quantum
grading of $X$ is two more than the quantum grading of $1$.

The differential on the Khovanov chain complex is a sum of maps along
the edges of the cube; it preserves the quantum grading and decreases
the homological grading by $1$. The component of the differential
along the edge from the vertex
$v=(v_1,\dots,v_{n-1},1,v_{n+1},\dots,v_N)$ to the vertex
$w=(v_1,\dots,v_{n-1},0,v_{n+1},\dots,v_N)$ is
$(-1)^{v_1+\dots+v_{n-1}}$ times the map associated by the TQFT to the
saddle cobordism. 
That is, if the saddle cobordism merges two circles
into one, then the map is induced by the multiplication map in
$\ZZ[X]/(X^2)$, and if the saddle cobordism splits a circle into
two, then the map is induced by the comultiplication map in $\ZZ[X]/(X^2)$.

This convention differs from Khovanov's
original~\cite{Kho-kh-categorification} in a couple of ways: the
differential in Khovanov's original paper increased the homological
grading; and the quantum grading of $X$ was lower than the quantum
grading of $1$. However, in order to ensure that arc algebras are
supported in non-negative quantum gradings, Khovanov switched the
latter convention in~\cite{Kho-kh-tangles}, and his subsequent papers
follow the switched convention (where the quantum grading of $X$ is
higher than the quantum grading of $1$). However, Khovanov's original
quantum grading convention had a desirable feature that the positive
knots (except the unknot) had Khovanov homologies supported in positive
quantum gradings; unfortunately, with the switched convention, their
Khovanov homologies were supported in negative quantum gradings. With our
grading conventions---additionally making the differential decrease
the homological grading---we try to tread a middle ground:
arc algebras and Khovanov homologies of (non-trivial)
positive knots are both supported in non-negative quantum
gradings.

The Khovanov chain complex with our convention is the dual of
the Khovanov chain complex from Khovanov's original convention,
preserving the bigrading (this follows easily from the duality
statement~\cite[Proposition~32]{Kho-kh-categorification}). So, over
any field, Khovanov homology with our convention is (unnaturally)
bigraded isomorphic to the original Khovanov homology; over $\ZZ$, the
free parts are bigraded isomorphic, and the torsion subgroup with our
convention is isomorphic to the original torsion subgroup, but with
its homological grading shifted down by $1$ (and quantum grading
unchanged).

The Khovanov complex of a link in the annulus inherits an extra
\emph{annular} or \emph{winding number} filtration; the homology of
the associated graded complex is annular Khovanov homology. We follow
the usual conventions in the literature for the annular
filtration. Specifically, given a labeled resolution, orient circles
labeled $1$ counter-clockwise (positively) and circles labeled $X$
clockwise (negatively). Then, the annular filtration of a labeled
resolution is the winding number around the axis. Terms in the
differential either preserve the annular filtration or decrease it by
$2$. (The latter occurs when merging a nullhomotopic circle labeled
$X$ with an essential circle labeled $1$, merging two essential
circles labeled $1$, splitting an essential circle labeled $1$ into a
nullhomotopic circle labeled $1$ and an essential circle labeled $X$,
or splitting a nullhomotopic circle labeled $X$ into two essential
circles labeled $X$.)

\section{A map on annular Khovanov homology}\label{sec:moving}
Let $L$ be an annular link diagram. Fix a point $p$ on $L$ adjacent to
the axis of the annulus. Isotoping $p$ across the axis of the annulus
gives a new link $L'$. (See Figure~\ref{fig:axis-moving}.) In this
section we define and spell out basic properties of the axis moving
maps
\[
f^+,\ f^-\co \AKh(L)\to \AKh(L').
\]
While the rest of the results in this paper use $\FF_2$-coefficients,
in this section, we will work with $\ZZ$-coefficients.

Let $L\amalg U$ be the result of adding an essential circle $U$ around
the axis, adjacent to the axis and disjoint from $L$. (See
Figure~\ref{fig:axis-moving}.) The annular Khovanov complex of
$L\amalg U$ is
$\AKhCx(L)\otimes \AKhCx(U)=\AKhCx(L)\otimes\ZZ\langle
1,X\rangle$. So, there are two inclusions
$\iota_1,\iota_X\co \AKhCx(L)\into \AKhCx(L\amalg U)$, defined by
$\iota_1(y)=y\otimes 1$ and $\iota_X(y)=y\otimes X$; these have
(homological, quantum, annular) trigradings $(0,-1,1)$ and $(0,1,-1)$,
respectively. Merging $L$ and $U$ at the point $p$ gives a map
$m\co \AKhCx(L\amalg U)\to L$; this map has trigrading
$(0,1,0)$. (This is the annular merge map. So, for instance, merging
two essential circles labeled $1$ is the zero map.) By composing, we get trigrading-preserving maps 
\begin{align*}
  f^+&=m\circ \iota_1 \co \Sigma^{0,0,1}\AKhCx(L)\to \AKhCx(L') \\
  f^-&=m\circ \iota_X \co \Sigma^{0,2,-1}\AKhCx(L)\to \AKhCx(L'),
\end{align*}
where $\Sigma^{a,b,c}$ denotes a trigrading shift by $(a,b,c)$.  The maps
$f^\pm$ are compositions of chain maps, hence are chain maps. Abusing
notation, $f^\pm$ also denote the induced map on annular homology.

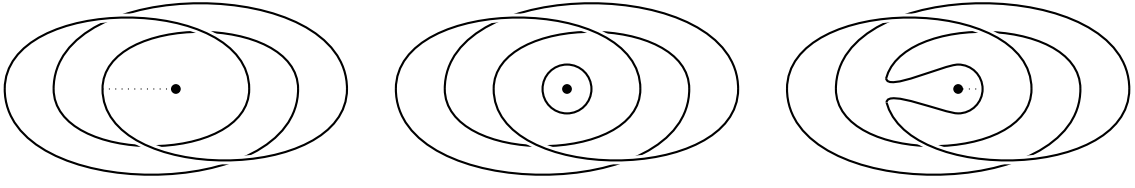
\begin{figure}
  \centering
  \begin{tikzpicture}[scale=.65]
    \foreach\i in {0,1,2}{
      \begin{scope}[xshift=8*\i cm]
    
        \coordinate (l1) at (0,0);
        \coordinate (l2) at (1,0);
        \coordinate (l3) at (2,0);
        \coordinate (l4) at (3,0);
        \coordinate (r0) at (4,0);
        \coordinate (r1) at (5,0);
        \coordinate (r2) at (6,0);
        \coordinate (r3) at (7,0);
        \coordinate (t) at (3.5,0.5);
        \coordinate (b) at (3.5,-0.5);
        \coordinate (axis) at (3.5,0);

        \draw[knot] (l2) to[out=90,in=90] (r3);
        \draw[knot] (l3) to[out=90,in=90] coordinate[pos=0.05] (l3t) (r2);
        \draw[knot] (l1) to[out=90,in=90] (r1);
        \draw[knot] (r1) to[out=-90,in=-90] (l2);
        \draw[knot] (r2) to[out=-90,in=-90] (l1);
        \draw[knot] (r3) to[out=-90,in=-90] coordinate[pos=0.95] (l3b) (l3);

        \fill[black] (axis) circle (0.1); 

        \ifnum\i=1
        \draw[knot] (l4) to[out=90,in=180] (t) to[out=0,in=90] (r0) to[out=-90,in=0] (b) to[out=180,in=-90] (l4);
        \fi

        \ifnum\i=0
        \draw[dotted](l3) to (axis);
        \fi

        \ifnum\i=2
        \draw[dotted](r0) to (axis);
        \draw[white,line width=10] (l3b)--(l3t);
        \draw[knot] (l3t) to[out=-90,in=180,looseness=0.5] (t) to[out=0,in=90] (r0) to[out=-90,in=0] (b) to[out=180,in=90,looseness=0.5] (l3b);
        \fi
        
      \end{scope}
    }
  \end{tikzpicture}  
  \caption{\textbf{The axis moving maps.} Left: the link $L$ and the
    arc connecting a point $p$ on $L$ to the axis (dotted). Center:
    the link $L\amalg U$. Right: the link $L'$ and the corresponding
    arc connecting $p'$ to the axis.}
  \label{fig:axis-moving}
\end{figure}

When we want to indicate the dependence of $f^\pm$ on the point $p$ we
will write them as $f^{\pm,p}$. Also, the point $p$ on $L$ becomes a
point $p'$ on $L'$ after the isotopy across the axis.

We give some basic properties of the maps $f^\pm$; though we will not
use them in the rest of this paper, perhaps they will be useful in
related applications.
\begin{proposition}
  Up to homotopy and sign, the maps $f^\pm$ are isotopy invariants of
  the pair of the annular link $L$ and the arc connecting $p$ to the
  axis.  Further,
  \begin{align}
   \label{eq:pp} f^+_{p'}\circ f^+_p &=f^-_{p'}\circ f^-_p=0 \\ 
    \label{eq:pn} f^+_{p'}\circ f^-_p&= f^-_{p'}\circ f^+_p= X\cdot_p.
  \end{align}
\end{proposition}
Here, $X\cdot_p$ is the basepoint action on annular Khovanov homology, the
result of merging in a nullhomotopic circle labeled $X$ at the point $p$, via
the annular merge map. Also, the first statement uses naturality of annular
Khovanov homology (see, e.g.,~\cite{GW-kh-Schur}).
\begin{proof}
  For the first statement, call an arc from an annular link to the
  axis \emph{short} if the projection to the annulus is a smooth
  embedding and disjoint from the projection of the rest of the link
  to the annulus (i.e., the rest of the link diagram).  Given a pair
  of an annular link $L$ and a (not short) arc $\gamma$ connecting
  that link to the axis, there is a canonical isotopy from
  $L\cup\gamma$ to a pair $L'\cup\gamma'$ where $\gamma'$ is short:
  just shrink $\gamma$ and pull $L$ along with it.

  Now, fix link diagrams $L_i$, $i=0,1$, and points $p_i$ on $L_i$
  adjacent to the axis, and let $\gamma_i$ be the short arc from $p_i$
  to the axis. Assume that $L_0\cup\gamma_0$ is isotopic to
  $L_1\cup\gamma_1$, via an isotopy $L_t\cup\gamma_t$. Applying the
  canonical isotopy $L_{t,s}\cup\gamma_{t,s}$ from the previous
  paragraph to each $L_t$ gives a new isotopy
  $L_{t,1}\cup\gamma_{t,1}$ from $L_0\cup\gamma_0$ to
  $L_1\cup\gamma_1$ so that the arc $\gamma_{t,1}$ is short for all
  $t$. Perturb this isotopy so that the projection to the annulus is
  generic. This gives a sequence of Reidemeister moves (in the
  annulus) connecting $L_0$ to $L_1$ and disjoint from the arc. Each
  Reidemeister moves induces a map on the annular Khovanov complex,
  and it is immediate from the definitions that these maps commute
  with the maps $f^\pm$.
  
  Equations~\eqref{eq:pp} and~\eqref{eq:pn} are clear
  from the definitions.
\end{proof}

\begin{remark}
  These maps are a special case of maps associated to \emph{anchored
    cobordisms} by
  Akhmechet-Khovanov~\cite{AK-kh-anchored}. Specifically, the trace of
  an isotopy moving $p$ across the axis is an anchored cobordism. To
  define a map, we must also label the intersection point between this
  cobordism and the axis (the point where $p$ crosses the axis) by $1$
  or $2$. If the winding number of $K$ around the axis is odd then
  $f^+$ corresponds to labeling the intersection point $1$ and $f^-$
  corresponds to labeling it $2$; if the winding number is even then
  $f^+$ corresponds to labeling the intersection point $2$ and $f^-$
  corresponds to labeling it $1$ (compare~\cite[Proof of Theorem
  2.19]{AK-kh-anchored}).
    
  Equations~\eqref{eq:pp} and~\eqref{eq:pn} follow from Akhmechet-Khovanov's curtain
  relation~\cite[Corollary 2.7]{AK-kh-anchored}.
\end{remark}

\begin{remark}\label{rem:kinda-annular}
  Given $(L,p)$ as above, let $L''$ be the result of performing a
  Reidemeister I move at $p$ across the axis, changing the winding
  number by $\pm 1$ (i.e., a Markov 2 move). Choose the Reidemeister
  move so that the $1$-resolution of the new crossing is the disjoint
  union of $L$ and a new (essential) circle. Then, up to some overall
  grading shift, the mapping cone of
  $f^+\co \Sigma^{0,0,1}\AKhCx(L)\to \AKhCx(L')$ (respectively
  $f^-\co \Sigma^{0,2,-1}\AKhCx(L)\to \AKhCx(L')$) is a subcomplex of
  $\AKhCx(L'')$: it is the subcomplex where either the new crossing is
  $0$-resolved or the new crossing is $1$-resolved and the new
  essential circle is labeled $1$ (respectively $X$).
\end{remark}

\section{The localization theorem for strongly invertible knots}\label{sec:proof}

In this section, we prove Theorem~\ref{thm:main}. The spectral
sequence is constructed as follows. Given an intravergent diagram for
$K$, rotation $\tau$ by $180^\circ$ induces a $\ZZ/2$ action, which we
still call $\tau$, on $\KhCx(K;\FF_2)$. Consider the Tate complex for this
$\ZZ/2$-action, which is given by
$\KhCx(K;\FF_2)\otimes\FF_2[\theta^{-1},\theta]$ with differential
$d+\theta(\Id+\tau)$, where $d$ is the Khovanov differential:
\begin{equation}\label{eq:tate}
  \vcenter{\hbox{
  \begin{tikzpicture}[xscale=3.5]
    \foreach\i in {-1,0,1}{
      \node (t\i) at (\i,0) {$\theta^{\i}\KhCx(K;\FF_2)$};
      \draw[->] (t\i) to[out=120,in=60,looseness=10] node[midway,anchor=south] {\small $d$} (t\i);
    }
    \node (t-2) at (-1.8,0) {$\cdots$};
    \node (t2) at (1.8,0) {$\cdots$.};

    \foreach \i [count=\j from -1] in {-2,-1,0,1}{
      \draw[->] (t\i)--(t\j) node[midway,anchor=south] {\small $\theta(\Id+\tau)$};
    }
  \end{tikzpicture}}}
\end{equation}
If we give $\theta$ the (homological, quantum) bigrading $(-1,0)$ then
this complex decomposes according to quantum gradings and the
differential decreases the homological grading by $1$. The complex has
a filtration by the $\theta$-power. This filtration induces the
spectral sequence in Theorem~\ref{thm:main}.

The main work is to compute the $E^\infty$-page of the spectral
sequence (Item~\ref{item:main-thm-e-infty}); the other properties are immediate. The
strategy is similar to Stoffregen-Zhang's~\cite{SZ-kh-localization} and
Borodzik-Politarczyk-Silvero's~\cite{BPS-kh-periodic}: we prove that the
fixed points of the $\ZZ/2$-action induced by the strong inversion on
a CW complex representing
the Khovanov stable homotopy type of $K$ is related to the annular
Khovanov stable homotopy type of $(\ol{K}_1,\ol{K}_0)$, and then apply
classical Smith theory.

Before embarking on the proof, we briefly summarize the relevant
aspects of the Khovanov stable homotopy type~\cite{RS-khovanov},
following the more recent box map
construction~\cite{LLS-khovanov-product,LLS-cube} (see
also~\cite{HKK-Kh-htpy}). As sketched in
Section~\ref{sec:conventions}, in a fixed quantum grading $j$, the
Khovanov chain complex $\KhCx_{*,j}(K)$ is obtained as follows:
\begin{enumerate}[label=(Kh-\arabic*),leftmargin=*]
\item\label{step:kh-cube-abelian} First construct a cube-shaped
  diagram of abelian groups,
  $F^{\Kh}_j\from(1\to 0)^N\to\AbelianGroups$, which associates to
  each vertex $v$ the free abelian group generated by all Khovanov
  generators $x$ at $v$ with quantum grading $j$, and associates
  to each edge $v\to w$ the saddle cobordism map from the TQFT
  corresponding to the Frobenius algebra $\ZZ[X]/(X^2)$.
\item\label{step:iterated-cone}\setcounter{dummycounter}{\value{enumi}}
  Viewing $F^{\Kh}_j$ as a cube-shaped diagram of chain complexes,
  take the mapping cone $N$ times, in the $N$ directions of
  the cube, to obtain a single chain complex. (The signs, like
  $(-1)^{v_1+\dots+v_{n-1}}$ from Section~\ref{sec:conventions}, appear
  during this iterated mapping cone construction; the precise signs depend on the
  order in which the mapping cones are done.)
\item Finally, shift the homological grading of the chain complex down
  by $N_-$ to get the Khovanov chain complex $\KhCx_{*,j}(K)$. (The
  total Khovanov complex over all quantum gradings is given by
  $\bigoplus_j\KhCx_{*,j}(K)$.)
\end{enumerate}
In order to remove the choice about the order in which to take the
mapping cones (which amounts to choosing an ordering of the $N$ crossings
of the knot diagram), one can replace Step~\ref{step:iterated-cone} by
the following:
\begin{enumerate}[label=(Kh-\arabic*${}'$),leftmargin=*,start=\value{dummycounter}]
\item Extend $F_j^{\Kh}$ trivially to a diagram $(F^{\Kh}_j)_+$ from
  a slightly larger category $(1\to 0)^N_+$ which has an additional
  object and a unique morphism from every $v\neq 0$ to it, and then
  take the homotopy colimit to obtain a single chain complex. (The
  functor $(F^{\Kh}_j)_+$ sends the new object to the trivial group.)
\end{enumerate}
The lift to a Khovanov stable homotopy type follows the same
outline, replacing the cube of abelian groups by a cube of
topological spaces.  More concretely:
\begin{enumerate}[label=($\KhSpace$-\arabic*),leftmargin=*,start=0]
\item Fix an integer $D\geq N$.
\item\label{step:kh-cube-space}\setcounter{dummycounter}{\value{enumi}}
  Construct a cube-shaped diagram
  $F^{\KhSpace}_j\from (1\to0)^N\to\Spaces$ of based CW complexes
  which associates to each vertex a wedge sum of $D$-dimensional spheres, so that
  its composition with the reduced homology functor,
  $\wt{H}_D\circ F^{\KhSpace}_j$, equals the cube-shaped diagram
  $F^{\Kh}_j$ from Step~\ref{step:kh-cube-abelian}. (It suffices to
  construct a homotopy coherent diagram instead of a strictly
  commuting one.)
\item Extend (trivially) to a diagram $(F^{\KhSpace}_j)_+$ from the
  larger category $(1\to 0)^N_+$ mapping the new object to a $1$-point
  space, and then take homotopy colimit to
  obtain a single topological space.
\item Finally, formally desuspend $(D+N_-)$ times to get the Khovanov
  spectrum $\KhSpace_j(K)$ in quantum grading $j$. (The total Khovanov
  spectrum over all quantum gradings is given by
  $\bigvee_j\KhSpace_j(K)$.)
\end{enumerate}
Of these, the hardest step is Step~\ref{step:kh-cube-space}, which we
undertake by first constructing a homotopy coherent diagram in a third
category---the Burnside 2-category $\BurnsideCat$ of finite sets,
finite correspondences, and bijections between
correspondences. Specifically, letting
$\ZZ\langle\cdot\rangle \from\BurnsideCat\to\AbelianGroups$ denote the
functor that replaces a finite set by the free abelian group generated
by it, we do the following:
\begin{enumerate}[label=($\BurnsideCat$-\arabic*),leftmargin=*]
\item Construct a cube-shaped 2-functor
  $F^\BurnsideCat_j\from (1\to 0)^N\to\BurnsideCat$ which associates
  to each vertex $v$ the set of Khovanov generators $x$ at $v$
  with quantum grading $j$, to each edge $v\to w$ a
  correspondence
  $F^\BurnsideCat_j(v)\stackrel{s}{\leftarrow} F^\BurnsideCat_j(v\to w)\stackrel{t}{\to} F^\BurnsideCat_j(w)$,
  so that $\ZZ\langle\cdot\rangle \circ F^{\BurnsideCat}_j=F^\Kh_j$,
  and to each 2-dimensional face
  $\vcenter{\hbox{\begin{tikzpicture}[xscale=1,yscale=0.2] \node (u)
        at (0,0) {$u$}; \node (v) at (1,1) {$v$}; \node (vp) at (1,-1)
        {$v'$}; \node (w) at (2,0) {$w$}; \draw[->] (u) -- (v);
        \draw[->] (u) -- (vp); \draw[->] (v)--(w); \draw[->] (vp) -- (w);
      \end{tikzpicture}}}$ a 2-morphism (which is an isomorphism of
  correspondences)
  $F^\BurnsideCat_j(v\to w)\circ F^\BurnsideCat_j(u\to
  v)\stackrel{\cong}{\lra} F^\BurnsideCat_j(v'\to w)\circ
  F^\BurnsideCat_j(u\to v')$ satisfying a coherence relation for every
  3-dimensional face. Note that this data specifies, for any $v>w$ in the
  poset $(1>0)^N$, a correspondence
  $F^\BurnsideCat_j(v\to w)\from F^\BurnsideCat_j(v)\to
  F^\BurnsideCat_j(w)$, by composing the correspondences along a
  sequence of oriented edges connecting $v$ to $w$; any two such
  sequences specify canonically isomorphic correspondences. (We can
  also define the total Burnside functor over all quantum gradings as
  $\coprod_j F^\BurnsideCat_j$ where $\amalg$ is defined by taking
  disjoints unions of sets and correspondences at vertices and edges,
  respectively.)
\end{enumerate}
Most of the construction of $F^\BurnsideCat_j$ is forced. The only
choices are the isomorphisms of correspondences for
certain 2-dimensional faces (which we call ladybugs), and we
explicitly choose the isomorphisms for those faces (which we call
ladybug matchings)~\cite[Section 5.4]{RS-khovanov},
\cite[Section 8.1]{LLS-khovanov-product}. Once we have the 2-functor
$F^\BurnsideCat_j\from(1\to0)^N\to \BurnsideCat$, we can carry out
Step~\ref{step:kh-cube-space} by the \emph{box map construction}, as follows:
\begin{enumerate}[label=($\KhSpace$-\arabic*${}'$),leftmargin=*,start=\value{dummycounter}]
\item\label{step:box-maps} For each vertex $v$, define
  $F^\KhSpace_j(v)=\big(\coprod_{x\in F^\BurnsideCat_j(v)}
  B_{x}\big)/\bdy$, where $B_{x}$ is a $D$-dimensional rectangular prism (\emph{box}) associated
  to the Khovanov generator $x$. For each edge $v\to w$, choose
  disjoint $D$-dimensional sub-boxes
  $\{B_b\}_{b\in F^\BurnsideCat_j(v\to w)}$ inside
  $\coprod_{x\in F^\BurnsideCat_j(v)} B_{x}$ so that each $B_b$ lies
  in $B_{s(b)}$; define the map
  $F^\KhSpace_j(v\to w)\from F^\KhSpace_j(v)\to F^\KhSpace_j(w)$ by
  sending each sub-box $B_b\subset B_{s(b)}$ to $B_{t(b)}$ by scaling and
  translation, and the complement of all these sub-boxes to the
  basepoint. Call maps as in the previous sentence \emph{box maps}. To extend this to a homotopy coherent diagram
  $F^\KhSpace_j$ on the entire cube, we specify, for every chain
  $v_\ell>\dots>v_0$ in the poset $(1>0)^N$, a
  $[0,1]^{\ell-1}$-parameter family of box maps
  $F^\KhSpace_j(v_\ell)\to F^\KhSpace_j(v_0)$ satisfying certain
  coherence conditions on its boundary, and refining the
  correspondence
  $F^\BurnsideCat_j(v_\ell\to v_0)\from F_j^\BurnsideCat(v_\ell)\to
  F_j^\BurnsideCat(v_0)$. By induction on $\ell$, and using the
  coherence conditions and the $2$-morphisms in the Burnside category,
  such a family of maps is already defined on the boundary
  $\bdy[0,1]^{\ell-1}$. Extend it to the entire cube $[0,1]^{\ell-1}$
  using $(D-2)$-connectedness of the space of labeled sub-boxes and
  the assumption that $D\geq N$.
\end{enumerate}

The constructions for annular Khovanov complexes and annular Khovanov
homotopy types mirror these definitions. There is an extra annular
grading, and we only consider maps that preserve that
grading. Therefore, in each (quantum, annular) bigrading $(j,k)$, we get
diagrams $\AF^\Kh_{j,k}\co (1\to 0)^N\to\AbelianGroups$, $\AF^\BurnsideCat_{j,k}\co(1\to0)^N\to
\BurnsideCat$, and $\AF^\KhSpace_{j,k}\co (1\to0)^N\to\Spaces$; their
extensions $(\AF^\Kh_{j,k})_+\co (1\to 0)^N_+\to\AbelianGroups$ and
$(\AF^\KhSpace_{j,k})_+\co (1\to0)^N_+\to\Spaces$; and the chain complex
$\AKhCx_{*,j,k}$ and the spectrum $\AKhSpace_{j,k}$. For
the pair $(\ol{K}_1,\ol{K}_0)$ of annular knots whose annular Khovanov
chain complex $\AKhCx(\ol{K}_1,\ol{K}_0)$ is defined as the mapping
cone of
$\Sigma^{0,0,1}\AKhCx(\ol{K}_1)\stackrel{f^+}{\to}\AKhCx(\ol{K}_0)$,
it still may be viewed as a subcomplex of another annular Khovanov
chain complex---see Remark~\ref{rem:kinda-annular}---and therefore, all
these constructions work for the pair $(\ol{K}_1,\ol{K}_0)$ as
well.

Given an intravergent diagram of $K$, the strong inversion induces
$\ZZ/2$-actions on these various objects as follows.
\begin{lemma}\label{lem:the-action}
  The $180^\circ$ rotation on the intravergent diagram of $K$ induces a
  $\ZZ/2$-action on the cube $(1\to 0)^N$ and an \emph{external
  $\ZZ/2$-action} on the 2-functor
  $F^\BurnsideCat_j(K)\from (1\to 0)^N\to\BurnsideCat$ in each quantum
  grading $j$, in the sense
  of Stoffregen-Zhang~\cite[Definition~3.4]{SZ-kh-localization}.
\end{lemma}

\begin{proof}
  The proof is similar to Stoffregen-Zhang's corresponding result for
  $2$-periodic knots~\cite[Proposition~6.4]{SZ-kh-localization}. The $180^\circ$
  rotation $\tau$ around the axis induces a $\ZZ/2$-action on the $N$
  crossings of $K$, which in turn induces a $\ZZ/2$-action (also
  denoted $\tau$) on the cube category $(1\to 0)^N$ after identifying
  it with $(1\to 0)^{\{\text{crossings of }K\}}$ by ordering the
  crossings. It also induces a $\ZZ/2$-action (still denoted $\tau$)
  on the set of all Khovanov generators in quantum grading $j$,
  sending $F^\BurnsideCat_j(v)$ to $F^\BurnsideCat_j(\tau
  v)$. Moreover, for each edge $v\to w$, it induces an
  isomorphism of correspondences
  \[
    \begin{tikzpicture}[xscale=2.5,yscale=1.2]
      \node (v) at (0,0) {$F^\BurnsideCat_j(v)$};
      \node (e) at (1,0) {$F^\BurnsideCat_j(v\to w)$};
      \node (w) at (2,0) {$F^\BurnsideCat_j(w)$};
      \node (tv) at (0,-1) {$F^\BurnsideCat_j(\tau v)$};
      \node (te) at (1,-1) {$F^\BurnsideCat_j(\tau v\to\tau w)$};
      \node (tw) at (2,-1) {$F^\BurnsideCat_j(\tau w)$};
      \draw[->] (e) -- (v) node[midway,tightnode,anchor=south] {\tiny $s$};
      \draw[->] (e) -- (w) node[midway,tightnode,anchor=south] {\tiny $t$};
      \draw[->] (te) -- (tv) node[midway,tightnode,anchor=north] {\tiny $s$};
      \draw[->] (te) -- (tw) node[midway,tightnode,anchor=north] {\tiny $t$};
      \draw[->] (v)--(tv) node[midway,tightnode,anchor=west] {\tiny $\tau$};
      \draw[->,dashed] (e)--(te) node[midway,tightnode,anchor=west] {\tiny $\tau$};
      \draw[->] (w)--(tw) node[midway,tightnode,anchor=west] {\tiny $\tau$};
    \end{tikzpicture}
  \]
  since for any Khovanov generators
  $x\in F^\BurnsideCat_j(v),y\in F^\BurnsideCat_j(w)$, the set
  $s^{-1}(x)\cap t^{-1}(y)$ has 0 or 1 elements. So the only thing to
  check is that $\tau$ respects the ladybug matchings across
  2-dimensional faces, which holds since the ladybug matching is
  invariant under planar isotopy, and in particular, the $180^\circ$
  rotation $\tau$.
\end{proof}

This $\ZZ/2$-action on
$F^\BurnsideCat_j(K)\from (1\to 0)^N\to\BurnsideCat$ has a \emph{fixed
point functor}
$(F^\BurnsideCat_j(K))^\tau$~\cite[Definition~3.11]{SZ-kh-localization}. The functor
$(F^\BurnsideCat_j(K))^\tau$ is defined on the fixed subcategory of the cube category
$(1\to0)^N$, which is itself isomorphic to the cube category
$(1\to0)^{\ol{N}+1}$ where $\ol{N}=(N-1)/2$ is the number of crossings of the
quotient diagram $\ol{K}$; $(F^\BurnsideCat_j(K))^\tau$ assigns to vertices and edges the
fixed subset of the $\tau$-action on the sets and correspondences,
respectively. It turns out that these fixed point functors are related
to the Burnside 2-functors associated to the pair of annular knots
$(\ol{K}_1,\ol{K}_0)$:

\begin{lemma}\label{lem:the-fixed-points}
  For any quantum grading $j$, the fixed point functor
  $(F^\BurnsideCat_j(K))^\tau$ is isomorphic to
  $\coprod_{\ol{\jmath},k\mid 2\ol{\jmath}+k=j-1+3\Delta}
  \AF^\BurnsideCat_{\ol{\jmath},k}(\ol{K}_1,\ol{K}_0)$, where
  $\Delta$ is as in Equation~\eqref{eq:Delta}.
\end{lemma}

\begin{proof}
  The proof is similar to the $2$-periodic
  case~\cite[Theorem~6.7]{SZ-kh-localization}. First, order the $\ol{N}$
  crossings of $\ol{K}$ arbitrarily. Then, at any vertex
  $v=(v_1,\dots,v_{\ol{N}+1})\in (1\to 0)^{\ol{N}+1}$, the set
  $\AF^\BurnsideCat_{\ol{\jmath},k}(\ol{K}_1,\ol{K}_0)(v)$ is defined to
  be the set of Khovanov generators of $\ol{K}_{v_{\ol{N}+1}}$ over the
  vertex $(v_1,\dots,v_{\ol{N}})$ in (quantum, annular) bigrading
  $(\ol{\jmath},k-v_{\ol{N}+1})$.

  Order the $N$ crossings of $K$ such that the crossing on the
  axis is ordered last and the quotient map
  $\{\text{other crossings of }K\}\to \{\text{crossings of }\ol{K}\}$
  is order-preserving. Then, there is an inclusion of cube categories
  $\iota\from (1\to0)^{\ol{N}+1}\to(1\to0)^N$ which sends the vertex
  $v=(v_1,\dots,v_{\ol{N}},v_{\ol{N}+1})$ to the vertex
  $\iota(v)=(v_1,v_1,\dots,v_{\ol{N}},v_{\ol{N}},v_{\ol{N}+1})\in(1\to0)^N$; the image is
  precisely the fixed subcategory of $(1\to0)^N$.

  We construct a natural isomorphism
  \[
    \eta\co \Bigl(\hspace{-1em}\coprod_{\substack{\ol{\jmath},k\\
        2\ol{\jmath}+k=j-1+3\Delta}}
    \hspace{-1em}\AF^\BurnsideCat_{\ol{\jmath},k}(\ol{K}_1,\ol{K}_0)\Bigr)
    \lra \Bigl((F^\BurnsideCat_j(K))^\tau \circ\iota\Bigr)
  \]
  between the two Burnside functors.
  For any vertex $v\in(1\to0)^{\ol{N}+1}$ and any Khovanov generator
  $x\in \AF^\BurnsideCat_{\ol{\jmath},k}(\ol{K}_1,\ol{K}_0)(v)$, let
  $\eta(x)$ be the Khovanov generator of $K$ over the vertex
  $\iota(v)$ which labels each circle in the $\iota(v)$-resolution of $K$
  by the same label that $x$ labels its quotient circle in the
  $(v_1,\dots,v_{\ol{N}})$-resolution of $\ol{K}_{v_{\ol{N}+1}}$. It
  is a straightforward calculation that $\eta(x)$ has quantum grading
  $j$ if and only if $(\ol{\jmath},k)$ satisfies
  $2\ol{\jmath}+k=j-1+3\Delta$. For the reader's convenience, we
  summarize the calculation below.

  In the $(v_1,\dots,v_{\ol{N}})$-resolution of the annular knot $\ol{K}_{v_{\ol{N}+1}}$,
  let $a_1$ and $a_X$ be the numbers of essential circles that $x$ labels by
  $1$ and $X$, respectively, and let $b_1$ and $b_X$ be the numbers of
  non-essential circles that $x$ labels by $1$ and $X$, respectively. Then,
  \begin{align*}
    \ol{\jmath}&=\ol{N}-3\ol{N}_-+(|v|-v_{\ol{N}+1})+(a_X+b_X)-(a_1+b_1)\\
    k-v_{\ol{N}+1}&=a_1-a_X\\
    2\ol{\jmath}+k+6\ol{N}_-+1&=(2\ol{N}+1)+(2|v|-v_{\ol{N}+1})+(a_X+2b_X)-(a_1+2b_1).
  \end{align*}
  In the $\iota(v)=(v_1,v_1,\dots,v_{\ol{N}},v_{\ol{N}},v_{\ol{N}+1})$-resolution of $K$,
  the number of circles labeled $1,X$ by the generator $\eta(x)$ is
  $(a_1+2b_1)$ and $(a_X+2b_X)$, respectively. Then,
  \[
    j+3N_-=N+|\iota(v)|+(a_X+2b_X)-(a_1+2b_1)=(2\ol{N}+1)+(2|v|-v_{\ol{N}+1})+(a_X+2b_X)-(a_1+2b_1).
  \]
  Therefore, $2\ol{\jmath}+k=j-1+3(N_--2\ol{N}_-)=j-1+3\Delta$.

  To
  make the notation less cumbersome, let $\AF^\BurnsideCat_{[\ol{\jmath},k]}=
  \coprod_{\ol{\jmath},k\mid2\ol{\jmath}+k=j-1+3\Delta}\AF^\BurnsideCat_{\ol{\jmath},k}$.

  Next, we must specify the natural isomorphism $\eta$ on the
  $1$-morphisms, that is, for any edge $v\to w$ in
  $\{1\to0\}^{\ol{N}+1}$, we must specify an isomorphism between
  correspondences:
  \[
    \AF^\BurnsideCat_{[\ol{\jmath},k]}(\ol{K}_1,\ol{K}_0)(v\to w)
    \stackrel{\eta}{\lra} F^\BurnsideCat_{j}(K)(\iota(v)\to\iota(w)).
  \]
  As in the proof of Lemma~\ref{lem:the-action}, for any 
  generators
  $x\in\AF^\BurnsideCat_{[\ol{\jmath},k]}(\ol{K}_1,\ol{K}_0)(v)$,
  $y\in\AF^\BurnsideCat_{[\ol{\jmath},k]}(\ol{K}_1,\ol{K}_0)(w)$, the set
  $s^{-1}(x)\cap t^{-1}(y)\subset
  \AF^\BurnsideCat_{[\ol{\jmath},k]}(\ol{K}_1,\ol{K}_0)(v\to w)$ has
  either 0 or 1 elements, and the set
  $s^{-1}(\eta(x))\cap t^{-1}(\eta(y))\subset
  F^\BurnsideCat_{j}(K)(\iota(v)\to\iota(w))$ also has either 0 or 1
  elements, correspondingly; therefore, the isomorphism $\eta$ between
  the correspondences is forced. This is checked by a direct case
  analysis: When $v_{\ol{N}+1}=w_{\ol{N}+1}$, then this is Stoffregen-Zhang's case
  analysis for the
  2-periodic link $K_{v_{\ol{N}+1}}$~\cite[Theorem~6.8]{SZ-kh-localization}. When
  $v_{\ol{N}+1}>w_{\ol{N}+1}$, it follows from a similar (but shorter)
  case analysis. Consider the axis-moving isotopy from the
  $(v_1,\dots,v_{\ol{N}})$-resolution of the annular knot $\ol{K}_1$
  to the corresponding resolution of the annular knot $\ol{K}_0$.
  There are two cases, depending on whether an essential circle
  becomes inessential or an inessential circle becomes essential. In
  the first (respectively second) case,
  $s^{-1}(x)\cap t^{-1}(w)\subset
  \AF^\BurnsideCat_{[\ol{\jmath},k]}(\ol{K}_1,\ol{K}_0)(v\to w)$ is
  non-empty (and has only one element) if and only if $x$ and $y$
  label the moving circle by $X$ (respectively $1$) and all other
  circles by the same labels. In the picture for $K$, we get a
  corresponding saddle cobordism from the
  $\iota(v)$-resolution of $K$ to the $\iota(w)$-resolution of $K$. In
  the first (respectively second) case, the saddle is a split
  (respectively merge) and
  $s^{-1}(\eta(x))\cap t^{-1}(\eta(y))\subset
  F^\BurnsideCat_{j}(K)(\iota(v)\to\iota(w))$ is non-empty (and has
  only one element) if and only if $\eta(x)$ and $\eta(y)$ label all
  the circles involved in the saddle by $X$ (respectively $1$), and
  all other circles by the same labels.

  Finally, we have to check that these isomorphisms of
  correspondences are compatible across $2$-dimensional faces. That
  is, given a 2-dimensional face
  $\vcenter{\hbox{\begin{tikzpicture}[xscale=1,yscale=0.2] \node (u)
        at (0,0) {$u$}; \node (v) at (1,1) {$v$}; \node (vp) at (1,-1)
        {$v'$}; \node (w) at (2,0) {$w$}; \draw[->] (u) -- (v);
        \draw[->] (u) -- (vp); \draw[->] (v)--(w); \draw[->] (vp) --
        (w);
      \end{tikzpicture}}}$ in $\{1\to0\}^{\ol{N}+1}$, we have to check that the following diagram commutes:
  \[
    \begin{tikzpicture}[xscale=8.5,yscale=-1.2]
      \node (af0) at (0,0) {$\AF^\BurnsideCat_{[\ol{\jmath},k]}(\ol{K}_1,\ol{K}_0)(v\to w)\circ \AF^\BurnsideCat_{[\ol{\jmath},k]}(\ol{K}_1,\ol{K}_0)(u\to v)$};
      \node (af1) at (0,1) {$\AF^\BurnsideCat_{[\ol{\jmath},k]}(\ol{K}_1,\ol{K}_0)(v'\to w)\circ \AF^\BurnsideCat_{[\ol{\jmath},k]}(\ol{K}_1,\ol{K}_0)(u\to v')$};
      \node (f0) at (1,0) {$F^\BurnsideCat_{j}(K)(\iota(v)\to\iota(w))\circ F^\BurnsideCat_{j}(K)(\iota(u)\to\iota(v))$};
      \node (f1) at (1,1) {$F^\BurnsideCat_{j}(K)(\iota(v')\to\iota(w))\circ F^\BurnsideCat_{j}(K)(\iota(u)\to\iota(v'))$};
      
      \draw[->] (af0) --(af1);
      \draw[->] (f0) --(f1);
      \draw[->] (af0) --(f0);
      \draw[->] (af1) --(f1);
    \end{tikzpicture}
  \]
  where the horizontal arrows are induced by the isomorphisms that we
  just constructed, and the vertical arrows are induced by the
  isomorphisms that are part of the data for the respective Burnside
  2-functors. Unless the 2-dimensional face is a ladybug, for any pair
  of generators
  $x\in\AF^\BurnsideCat_{[\ol{\jmath},k]}(\ol{K}_1,\ol{K}_0)(u),
  y\in\AF^\BurnsideCat_{[\ol{\jmath},k]}(\ol{K}_1,\ol{K}_0)(w)$, both
  the sets
  $s^{-1}(x)\cap t^{-1}(w)\subset
  \AF^\BurnsideCat_{[\ol{\jmath},k]}(\ol{K}_1,\ol{K}_0)(u\to w)$ and
  $s^{-1}(\eta(x))\cap t^{-1}(\eta(y))\subset
  F^\BurnsideCat_{j}(K)(\iota(u)\to\iota(w))$ have 0 or 1 elements,
  and so the check is automatic. So the only case remaining is when
  the 2-dimensional face is a ladybug. However, recall from
  Remark~\ref{rem:kinda-annular} that the functor
  $\AF^\BurnsideCat_{[\ol{\jmath},k]}(\ol{K}_1,\ol{K}_0)$ may be viewed
  as subfunctor of the Burnside functor associated to a different
  annular knot obtained from $\ol{K}_1$ by performing a Reidemeister I
  move. However, such a crossing (coming from a Reidemeister I move)
  cannot be involved in a ladybug configuration. Therefore, in order
  to be a ladybug, we must have
  $u_{\ol{N}+1}=v_{\ol{N}+1}=v'_{\ol{N}+1}=w_{\ol{N}+1}$; and in that case, the
  commutation of the above diagram follows
  from the analogue for the 2-periodic link
  $K_{u_{\ol{N}+1}}$~\cite[Lemma~6.15]{SZ-kh-localization}.
\end{proof}

Now, given the $\ZZ/2$-action on the Burnside functor
$F^\BurnsideCat_j(K)$, and the above identification of its fixed point
functor with those of the quotient annular knots
$(\ol{K}_1,\ol{K}_0)$, all that remains is to refine these actions and
the fixed points to the category of topology spaces. This is precisely
Stoffregen-Zhang's central thesis:

\begin{proposition}\cite[Proposition~5.10]{SZ-kh-localization}\label{prop:equivariant-box-maps}
  Let $F^\BurnsideCat\from (1\to0)^N\to\BurnsideCat$ be a Burnside
  2-functor with an external $\ZZ/2$-action $\tau$ and
  $(F^\BurnsideCat)^\tau$ denote the fixed point functor. Then, the
  homotopy coherent diagram $F^\KhSpace\from(1\to0)^N\to\Spaces$
  refining $F^\BurnsideCat$ using the box map construction, as in
  Step~\ref{step:box-maps}, may be chosen $\ZZ/2$-equivariantly so
  that the fixed point homotopy coherent diagram $(F^\KhSpace)^\tau$
  refines $(F^\BurnsideCat)^\tau$ using box maps.
\end{proposition}

\begin{proof} For the reader's convenience, we sketch the proof 
(summarizing the proofs of~\cite[Lemma~4.7 and Proposition~5.10]{SZ-kh-localization}).

  Fix $D_1,D_2\geq N$. For every $x\in\coprod_{v}F^\BurnsideCat(v)$,
  associate a $(D_1+D_2)$-dimensional box $B_x\cong[0,1]^{D_1+D_2}$;
  endow it with the $\ZZ/2$-action $\tau$ which reflects the
  first $D_1$-coordinates and is the identity along the last
  $D_2$-coordinates, i.e.,
  $\tau(x_1,\dots,x_{D_1+D_2})=(1-x_1,\dots,1-x_{D_1},x_{D_1+1},\dots,x_{D_1+D_2})$.

  As in Step~\ref{step:box-maps}, for chains $v_\ell>\dots>v_0$ in the
  poset $(1>0)^N$, we will construct a $[0,1]^{\ell-1}$-parameter family
  of box maps $F^\KhSpace_j(v_\ell)\to F^\KhSpace_j(v_0)$ refining the
  correspondence
  $F^\BurnsideCat(v_\ell\to v_0)\from F^\BurnsideCat(v_\ell)\to
  F^\BurnsideCat(v_0)$ by induction on $\ell$. These maps will
  already be specified on the boundary $\bdy[0,1]^{\ell-1}$ by the
  compatibility condition.  There are two cases:
  \begin{itemize}[leftmargin=*]
  \item If the entire chain $v_\ell>\dots>v_0$ is not fixed by $\tau$,
    choose one of the two chains $\mf{c}=(v_\ell>\dots>v_0)$ or
    $\tau\mf{c}=(\tau v_\ell>\dots>\tau v_0)$ arbitrarily; without loss of
    generality, say we pick $\mf{c}$. Construct the
    $[0,1]^{\ell-1}$-parameter family of box maps for
    $\mf{c}$, refining the correspondence
    $F^\BurnsideCat(v_\ell\to v_0)$, arbitrarily using the
    $(D_1+D_2-2)$-connectedness of the space of labeled sub-boxes.
    Then define the $[0,1]^{\ell-1}$-parameter family of
    box maps for the other chain $\tau\mf{c}$, refining the
    correspondence $F^\BurnsideCat(\tau v_\ell\to \tau v_0)$, by
    pre-composing and post-composing by $\tau$, as well as relabeling
    the sub-boxes by the map
    $\tau\from F^\BurnsideCat(v_\ell\to v_0)\to F^\BurnsideCat(\tau
    v_\ell\to \tau v_0)$.
  \item If the entire chain $v_\ell>\dots>v_0$ is fixed by $\tau$,
    construct the $[0,1]^{\ell-1}$-parameter family of box maps
    refining the correspondence $F^\BurnsideCat(v_\ell\to v_0)$ as follows.
    \begin{itemize}[leftmargin=*]
    \item Let $A\subset F^\BurnsideCat(v_\ell\to v_0)$ be the subset
      not fixed by $\tau$. From every pair $\{a,\tau a\}\subset A$,
      choose one element arbitrarily. Let $B\subset A$ be the subset of chosen
      elements. Pick the $[0,1]^{\ell-1}$-parameter family of
      sub-boxes labeled by $B$ in the complement of the $\tau$-fixed
      subspace of the boxes
      using the $(D_1-2)$-connectedness of
      that space. Construct the $[0,1]^{\ell-1}$-parameter of
      sub-boxes labeled by $A\setminus B$ by applying $\tau$.
    \item Let $C\subset F^\BurnsideCat(v_\ell\to v_0)$ be the subset
      fixed by $\tau$. Pick the $[0,1]^{\ell-1}$-parameter family of
      sub-boxes labeled by $C$ symmetrically with respect to
      $\tau$. (First choose a $[0,1]^{\ell-1}$-family of
      $D_2$-dimensional boxes inside the fixed subset
      $\{\tfrac{1}{2}\}^{D_1}\times[0,1]^{D_2}$ using the
      $(D_2-2)$-connectedness of that space, and then thicken the
      boxes $\tau$-equivariantly to get $(D_1+D_2)$-dimensional boxes,
      while staying disjoint from the sub-boxes labeled by $A$.)
    \end{itemize}
  \end{itemize}

  This produces a homotopy coherent diagram $F^\KhSpace$ refining
  $F^\BurnsideCat$ using $(D_1+D_2)$-dimensional box maps, and the
  fixed point functor $(F^\KhSpace)^\tau$ is also a homotopy coherent
  diagram refining $(F^\BurnsideCat)^\tau$ using $D_2$-dimensional
  box maps.
\end{proof}

Combining these ingredients, we get:
\begin{proposition}\label{prop:main-fixed-pts}
  The strong inversion of $K$ induces a $\ZZ/2$-action on the Khovanov
  spectrum $\KhSpace(K)$ whose geometric fixed point set is
  $\AKhSpace(\ol{K}_1,\ol{K}_0)$ up to some formal
  (de)suspension. Keeping track of quantum gradings, the geometric fixed point set
  of the $\ZZ/2$-action on $\KhSpace_j(K)$ is the spectrum
  $\bigvee_{\ol{\jmath},k\mid
    2\ol{\jmath}+k=j-1+3\Delta}\AKhSpace_{\ol{\jmath},k}(\ol{K}_1,\ol{K}_0)$,
  up to some formal (de)suspension.
\end{proposition}

\begin{proof}
  Choose the homotopy coherent diagram $F^\KhSpace_j(K)$
  $\ZZ/2$-equivariantly, as in
  Proposition~\ref{prop:equivariant-box-maps}. Up to some
  (de)suspension, the Khovanov spectrum $\KhSpace_j(K)$ is the
  homotopy colimit of the extended diagram $(F^\KhSpace_j(K))_+$. The
  geometric fixed point set of this homotopy colimit is the homotopy
  colimit of the extended fixed point functor
  $(F^\KhSpace_j(K))^\tau_+$. But by
  Lemma~\ref{lem:the-fixed-points} and
  Proposition~\ref{prop:equivariant-box-maps}, the fixed point functor
  $(F^\KhSpace_j(K))^\tau$ refines Burnside functor
  $\coprod_{\ol{\jmath},k\mid
    2\ol{\jmath}+k=j-1+3\Delta}\AF^\BurnsideCat_{\ol{\jmath},k}(\ol{K}_1,\ol{K}_0)$,
  and therefore, up to some (de)suspension, the homotopy colimit of
  $(F^\KhSpace_j(K))^\tau_+$ is simply
  $\bigvee_{\ol{\jmath},k\mid
    2\ol{\jmath}+k=j-1+3\Delta}\AKhSpace_{\ol{\jmath},k}(\ol{K}_1,\ol{K}_0)$,
  as claimed.
\end{proof}

\begin{proof}[Proof of Theorem~\ref{thm:main}]
  Since the spectral sequence is induced by the $\theta$-filtration on the Tate
  complex, Formula~(\ref{eq:tate}), we only need to prove
  Item~\ref{item:main-thm-e-infty}: the other
  parts are immediate from the definition. However, this is simply the
  classical Smith inequality applied to the
  Proposition~\ref{prop:main-fixed-pts}, stated in the 
  language of spectral sequences. To wit, the Tate complex from
  Equation~\eqref{eq:tate} computes the localized equivariant homology
  of $\KhSpace(K)$, which by the classical localization theorem,
  equals the localized equivariant homology of the geometric fixed
  point set $\AKhSpace(\ol{K}_1,\ol{K}_0)$, which simply equals its
  homology $\AKh(\ol{K}_1,\ol{K}_0)$, tensored with
  $\FF_2[\theta^{-1},\theta]$. (It is also easy to keep track of the
  quantum gradings using Proposition~\ref{prop:main-fixed-pts}.)
\end{proof}

\begin{remark}\label{rem:gr-shift}
  The expression $\Delta=N_--2\ol{N}_-$ appears as grading shifts in
  Theorem~\ref{thm:main}, but it is not an invariant of the knot $K$
  and its strong inversion. Geometrically, the 2-periodic annular
  links $K_0,K_1$ obtained by resolving the crossing of $K$ on the
  axis, and their quotient annular knots $\ol{K}_0,\ol{K}_1$, are only
  well-defined up to how many times they wind around the axis, and
  $\Delta$ captures information about this winding
  number. In more detail, if $B$ (respectively $T$) denotes the
  underpass (respectively overpass) of $K$ near its crossing on the
  axis, then orient the quotient knot $\ol{K}$ by orienting the
  quotient arc $\ol{B}$ (respectively $\ol{T}$) towards (respectively
  away from) the axis. This induces orientations of the two annular
  knots $\ol{K}_0$ and $\ol{K}_1$, as well as their pre-images
  $K_0$ and $K_1$ (but not of the original knot $K$). Let
  $W$ be the winding number of $K_0$ (equivalently, $\ol{K}_0$) around
  the axis; this is one higher than the winding number of $K_1$
  (equivalently, $\ol{K}_1$) around the axis. Then $W-\Delta$ is an
  invariant of the knot $K$ and its strong inversion; we prove this as
  Proposition~\ref{prop:classical-invt} below.
\end{remark}

\begin{proposition}\label{prop:classical-invt}
  The quantity $W-\Delta$ is independent of the choice of the
  intravergent diagram, and in fact equals twice the axis linking number
  invariant~\cite[Definition~4.6]{BI-top-invert-g4} of the knot $K$
  and its strong inversion.
\end{proposition}

\begin{proof}
  We first prove that $W-\Delta$ is an invariant. Given two
  intravergent diagrams for $K$ and its strong inversion, connect them
  by a generic $\ZZ/2$-equivariant isotopy in $\RR^3$; this produces a
  generic isotopy (in $\RR^3$) connecting the quotient diagrams for
  $\ol{K}$. As in the proof of Reidemeister's theorem, this implies
  that the two diagrams for $\ol{K}$ are related by a finite sequence of
  Reidemeister moves (away from the axis), as well as a new move,
  corresponding to the situation when during the isotopy of $K$, the
  projection of the underpass $B$ becomes tangent to the projection of
  the overpass $T$ at the axis. The Reidemeister moves of $\ol{K}$
  lift to $\ZZ/2$-equivariant pairs of Reidemeister moves for $K$,
  while this new move lifts to the move shown in the left half of
  Figure~\ref{fig:intravergent-moves}. (Actually, there are two moves,
  depending on whether the overpass $T$ rotates clockwise or
  counter-clockwise over $B$. Figure~\ref{fig:intravergent-moves}
  shows the move for the counter-clockwise rotation; the other move
  can be obtained by performing a $\ZZ/2$-equivariant pair of
  Reidemeister II moves of $T$ over $B$ near the axis, and then the
  above move in reverse.)

  \begin{figure}
    \centering
    \begin{tikzpicture}[scale=0.5]
      \foreach \i in {0,1}{
        \foreach \j in {0,1}{

          \begin{scope}[xshift=15*\j cm +7*\i cm]
            \coordinate (w) at (-2,0);
            \coordinate (e) at (2,0);
            \coordinate (n) at (0,2);
            \coordinate (s) at (0,-2);
            
            \coordinate (ww) at (-1,0);
            \coordinate (ee) at (1,0);
            \coordinate (nn) at (0,1);
            \coordinate (ss) at (0,-1);
            
            \coordinate (sw) at (-0.5,-0.8);
            \coordinate (ne) at (0.5,0.8);
            \coordinate (o) at (0,0);
            
            \ifnum\i=0

            \draw[->] (3,0)--(4,0);

            \ifnum\j=0
            \draw[knot] (w)--(e);
            \draw[knot] (n)--(s);
            \else
            \draw[knot] (w)--(ww) to[out=0,in=90] (ss);
            \draw[knot,-latex](ss) -- (s);
            \draw[knot] (e) -- (ee) to[out=180,in=-90] (nn);
            \draw[knot,-latex] (nn)--(n);
            \fi

            \else

            \ifnum\j=0
            \draw[knot] (w)--(e);
            \draw[knot] (n) to[out=-90,in=90] (ww) to[out=-90,in=180] (sw) to[out=0,in=-90] (o) to[out=90,in=180] (ne) to[out=0,in=90] (ee) to[out=-90,in=90] (s);
            \else
            \draw[thick] (w)--(ww) to[out=0,in=0,looseness=2] (sw);
            \draw[knot] (sw) to[out=180,in=-90] (ww);
            \draw[knot,-latex](ww) to[out=90,in=-90] (n);
            \draw[thick] (e)--(ee) to[out=180,in=180,looseness=2] (ne);
            \draw[knot] (ne) to[out=0,in=90] (ee);
            \draw[knot,-latex](ee) to[out=-90,in=90] (s);
            \fi
                        
            \fi
            
          \end{scope}
          
        }}
      
    \end{tikzpicture}
    \caption{\textbf{An additional move.} The left half shows a new
      move for intravergent diagrams, corresponding to rotating the
      overpass $T$ counter-clockwise over the underpass $B$. The right
      half shows the corresponding change for the 2-periodic link
      $K_0$ (with the orientation from
      Remark~\ref{rem:gr-shift}).}\label{fig:intravergent-moves}
  \end{figure}
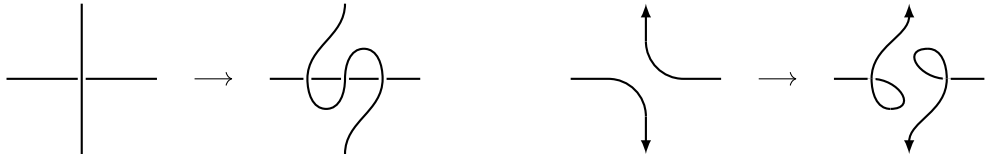

  The Reidemeister moves for $\ol{K}$---lifting to a
  $\ZZ/2$-equivariant pair of Reidemeister moves for $K$---do not
  change $W$, the winding number of $K_0$. For the Reidemeister I
  move, depending on the shape of the clasp, either $N_-$ increases by
  2 and $\ol{N}_-$ increases by 1, or both $N_-$ and $\ol{N}_-$ are
  unchanged. For the Reidemeister II move, $N_-$ increases by 2 and
  $\ol{N}_-$ increases by 1, and for the Reidemeister III move, both
  $N_-$ and $\ol{N}_-$ are unchanged. So, in each case
  $\Delta=N_--2\ol{N}_-$ does not change.  Finally, for the new move
  from Figure~\ref{fig:intravergent-moves}, $N_-$ increases by 1 and
  $\ol{N}_-$ also increases by 1, so $\Delta$ decreases by 1, but the
  winding number $W$ decreases by 1, so the quantity $W-\Delta$ is
  preserved.

  Next we will prove that this invariant $W-\Delta$ equals twice Boyle-Issa's
  axis linking number invariant~\cite[Definition~4.6]{BI-top-invert-g4}. Fix an orientation of
  the knot $K$. By performing the move from
  Figure~\ref{fig:intravergent-moves} once if necessary, we may assume
  the crossing of $K$ on the axis is a positive crossing. Then $K_0$
  is a 2-component link, and it inherits an orientation from $K$. To
  avoid confusion, let $o_{\mathrm{can}}$ denote the canonical
  orientation of $K_0$ from Remark~\ref{rem:gr-shift}, and let
  $o_{\mathrm{ind}}$ denote the induced orientation from $K$. These
  two orientations agree on one of the components of $K_0$, and
  disagree on the other. 

  The number of negative crossings of $K_0$ with orientation
  $o_{\mathrm{ind}}$ is $N_-$ and the number of the negative
  crossings of $K_0$ with orientation $o_{\mathrm{can}}$ is
  $2\ol{N}_-$. Therefore, $\Delta=N_--2\ol{N}_-$ is twice the
  linking number between the two components of $K_0$ (with orientation
  $o_{\mathrm{can}}$); in particular, it is an even number.

  Now perform the move from Figure~\ref{fig:intravergent-moves}
  $\Delta$ times. (If $\Delta<0$, then perform the reverse move
  $-\Delta$ times.) In the new diagram, the crossing on the axis is
  still positive, so the above discussion applies. Now $\Delta=0$, and
  so the invariant is simply the new winding number $W$. Also, the
  linking number between the two components of $K_0$ in the new
  diagram is zero, so $K_0$ is the 2-component butterfly
  link~\cite[Definition~4.1]{BI-top-invert-g4}, and by definition its
  winding number $W$ is twice the axis linking number invariant.
\end{proof}

\begin{remark}
  Given a theorem about Khovanov homology, it is natural to wonder if
  it lifts a result about the Jones polynomial.  Let $V_K(q)$ be the
  unreduced Jones polynomial, that is, the graded Euler characteristic
  of $\Kh(K)$. Let $J_K(q)=V_K(q)/(q+q^{-1})$ denote the reduced Jones
  polynomial. For an annular knot $K$, let $\AV_K(q,a)$ be the graded
  Euler characteristic of annular Khovanov homology, which was studied briefly by Roberts~\cite[Section
  2]{Roberts-kh-dcov}.
  By Theorem~\ref{thm:main},
  \[
  V_K(q)\equiv q^{1-3\Delta}(q\AV_{\ol{K}_1}(q^2,q)+\AV_{\ol{K}_0}(q^2,q))
  \pmod{2}.    
  \]
  It is easy to see from Kauffman's state sum formula that
  if we quotient by $(q^2+q^{-2})-(q+q^{-1})$ then $\AV(q^2,q)\equiv
  V(q^2)$ and   
  $V(q)\equiv V(q^2)$. Therefore, we have
  \[
  V_K(q)\equiv q^{-3\Delta}(q^2+q)V_{\ol{K}}(q^2) \equiv q^{-3\Delta}(q^2+q)V_{\ol{K}}(q)
  \pmod{2, q^2-q-q^{-1}+q^{-2}},
  \]
  where $\ol{K}$ denotes either $\ol{K}_0$ or $\ol{K}_1$, viewed as an
  ordinary, not annular, knot. Since
  $q^2-q-q^{-1}+q^{-2}=(q+q^{-1})(q-1+q^{-1})$ over $\FF_2[q^{-1},q]$,
  we may divide by $(q+q^{-1})$ to get the equation for reduced Jones
  polynomial
  \begin{equation}\label{eq:vacuous}
    J_K(q)\equiv q^{-3\Delta}(q^2+q)J_{\ol{K}}(q^2)\equiv q^{-3\Delta}(q^2+q)J_{\ol{K}}(q)\equiv J_{\ol{K}}(q)\pmod{2,q-1+q^{-1}}.
  \end{equation}
  An analogous result can also be obtained for 2-periodic knots using
  \cite[Theorem 1.3]{SZ-kh-localization}, giving the 2-periodic case
  of a formula of Murasugi's~\cite[Theorem
  1]{Murasugi-knot-per-Jones} and, using the fact that $J_K(i)\equiv
  1\pmod{2}$, the 2-periodic case of Yokota's refinement~\cite[Theorem
  2]{Yokota-knot-periodic}. However, Formula~\eqref{eq:vacuous} is actually
  vacuous, since if we quotient by $(q+q^{-1})-1$, in Kauffman's
  state sum formula each circle contributes $1$, and so for any knot
  or link diagram $K$ with $N$ crossings, $N_-$ of which are negative, we get
  \[
    J_K(q)\equiv \sum_{v\in\{0,1\}^N}q^{N+|v|-3N_-}=(1+q)^Nq^Nq^{-3N_-}\equiv
    1\pmod{2,q-1+q^{-1}}.
  \]
  (Murasugi's and Yokota's formulas are also vacuous for 2-periodic
  knots, though interesting for higher periods. For Murasugi, this
  is~\cite[Proposition 7]{Murasugi-knot-per-Jones}; Yokota only states
  his results for odd primes, presumably for this reason.)
\end{remark}

\section{An application to slice disks}\label{sec:slice}

\begin{figure}
  \centering
  \begin{tikzpicture}

    \node[circle, outer sep=1.5cm] (main0) at (0,0) {};
    \node[circle, outer sep=1.5cm] (main1) at (30:4.5) {};
    \node[circle, outer sep=1.5cm] (main2) at (-30:4.5) {};

    \path (main1) ++(3,0) node[outer sep=0.5cm] (main3) {} ++(2,0) node[outer sep=0.5cm] (main5) {};
    \path (main2) ++(3,0) node[outer sep=0.5cm] (main4) {} ++(2,0) node[outer sep=0.5cm] (main6) {};
    
    \foreach \c in {3,4}{
      
      \node at (main\c) {\begin{tikzpicture}[scale=0.4]

          \draw[knot] (0,0) circle (1);
          \draw[knot] (0,2.5) circle (1);
          
        \end{tikzpicture}};}

    \foreach \c in {5,6}{
      \node at  (main\c) {$\varnothing$};}
    
    \foreach \c in {0,1,2}{
      \node[tighternode] at (main\c) {\begin{tikzpicture}[scale=0.4]

          \ifnum\c=0
          \draw[dashed] (0,-1)--(0,6);
          \fi
          
          \node[tighternode] at (-4.5,0) {};
          \node[tighternode] at (4.5,0) {};

          \coordinate (b) at (0,5.5);
          \coordinate (t) at (0,4.5);
          
          \foreach \a in {0,1,2,3}{
            \foreach \b [count=\bi from 0] in {-2.5,-1.5,-0.5,0.5,1.5,2.5}{

              \coordinate (m\bi\a) at (\b,\a);

            }}

          \draw[knot] (m01) to[out=-90,in=180] ($(m10)!0.5!(m20)$) to[out=0,in=-90] (m31);
          \draw[knot] (m50) to[out=180,in=-90] (m41);

          \draw[knot] (m50) to[out=0,in=0,looseness=1.1] (b) to[out=180,in=180,looseness=1.1] (m00) to[out=0,in=-90] (m11);
          \draw[knot] (m21) to[out=-90,in=180] ($(m30)!0.5!(m40)$) to[out=0,in=-90] (m51);
          
          \draw[knot] (m11) to[out=90,in=-90] (m02);
          \draw[knot] (m21) to[out=90,in=-90] (m32);
          \draw[knot] (m51) to[out=90,in=-90] (m42);

          \draw[knot] (m01) to[out=90,in=-90] (m12);
          \draw[knot] (m31) to[out=90,in=-90] (m22);
          \draw[knot] (m41) to[out=90,in=-90] (m52);

          \ifnum\c=0
          \draw[knot] (m12) to[out=90,in=-90] (m03) to[out=90,in=180] (t);
          \draw[knot] (m22) to[out=90,in=180] ($(m33)!0.5!(m43)$) to[out=0,in=90] (m52);
          \draw[knot] (m02) to[out=90,in=180] ($(m13)!0.5!(m23)$) to[out=0,in=90] (m32);
          \draw[knot] (m42) to[out=90,in=-90] (m53) to[out=90,in=0] (t);
          \fi

          \ifnum\c=1
          \draw[knot] (m12) to[out=90,in=-90] (m03) to[out=90,in=150] ($(m33)+(0,1.3)$) to[out=-30,in=0] (m33);
          \draw[knot] (m22) to[out=90,in=180] (m33) (m43) to[out=0,in=90] (m52);
          \draw[knot] (m02) to[out=90,in=180] ($(m13)!0.5!(m23)$) to[out=0,in=90] (m32);
          \draw[knot] (m42) to[out=90,in=-90] (m53) to[out=90,in=0] ($(m43)+(0,1.3)$) to[out=180,in=180] (m43);
          \fi

          \ifnum\c=2
          \draw[knot] (m12) to[out=90,in=-90] (m03) to[out=90,in=180] ($(m13)+(0,1.3)$) to[out=0,in=0] (m13);
          \draw[knot] (m22) to[out=90,in=180] ($(m33)!0.5!(m43)$) to[out=0,in=90] (m52);
          \draw[knot] (m02) to[out=90,in=180] (m13) (m23) to[out=0,in=90] (m32);
          \draw[knot] (m42) to[out=90,in=-90] (m53) to[out=90,in=30] ($(m23)+(0,1.3)$) to[out=-150,in=180] (m23);
          \fi

        \end{tikzpicture}};}

    \draw[-latex] (main0)--(main1) node[pos=0.5,anchor=south,rotate=30] {\tiny saddle};
    \draw[-latex] (main0)--(main2) node[pos=0.5,anchor=north,rotate=-30] {\tiny saddle};

    \draw[-latex] (main1)--(main3) node[midway,anchor=south] {\tiny isotopy};
    \draw[-latex] (main2)--(main4) node[midway,anchor=north] {\tiny isotopy};

    \draw[-latex] (main3)--(main5) node[midway,anchor=south] {\tiny deaths};
    \draw[-latex] (main4)--(main6) node[midway,anchor=north] {\tiny deaths};
    
    \end{tikzpicture}        
    \caption{\textbf{The knot $K=9_{46}$ and a pair of slice disks for
        it.}  The knot is on the left, and the two movies on the two
      rows represent its two slice disks. Note that the two movies are
      related by a $180^\circ$ rotation around the dashed line.}
  \label{fig:9-46-slices}
\end{figure}
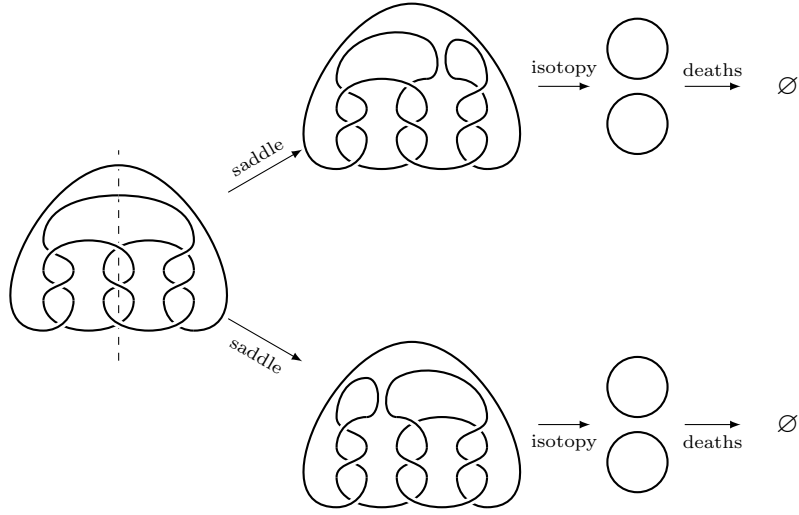

Consider the knot $K=9_{46}$. It bounds two slice disks as illustrated
in Figure~\ref{fig:9-46-slices}; denote them $D_1$ and $D_2$, and view
them as cobordisms in $[0,1]\times\RR^3$ from $K$ to the empty link
$\varnothing$. Let $\widehat{D}_i$
denote the image of $D_i$ under the map $(t,x,y,z)\mapsto (1-t,x,y,z)$, 
so $\widehat{D}_i$ is a cobordism from $\varnothing$ to
$K$. Sundberg-Swann showed that the disks $D_1$ and $D_2$ are
distinguished by their induced maps on Khovanov homology. We will
recover this result using Theorem~\ref{thm:main}. In fact, we get a
little more; see Porism~\ref{porism} below. The argument is
reminiscent of the recent work of Dai-Mallick-Stoffregen using
Heegaard Floer homology~\cite{DMS-hf-invert}.

\begin{theorem}\cite{SS-kh-surf}\label{thm:slice}
  The slice disks $D_1$ and $D_2$ induce different maps on Khovanov
  homology $\Kh(9_{46};\FF_2)\to\Kh(\varnothing;\FF_2)=\FF_2$.
\end{theorem}

\begin{proof}
  For any cobordism $F$, let $F_*$ denote the induced map on Khovanov
  homology. We will find an element $\gamma\in\Kh_{0,1}(K;\FF_2)$
  satisfying $(D_1)_*(\gamma)\neq (D_2)_*(\gamma)$.
  
  The Khovanov homology of $K$ in quantum grading $\pm 1$,
  retrieved from the Knot Atlas~\cite{KAT-kh-knotatlas} and converted
  to the conventions of Section~\ref{sec:conventions}, is
\[
\begin{tikzpicture}[xscale=1.2,yscale=0.35]
  \foreach \i in {-1.5,-0.5,0.5}{
    \draw (\i,-2)--(\i,2);}
  \foreach \j in {-2,0,2}{
    \draw (-1.5,\j)--(0.5,\j);}
  \foreach \i in {-1,0}{
    \node[anchor=south] at (\i,2) {$\i$};}
  \foreach \j in {-1,1}{
    \node[anchor=east] at (-1.5,\j) {$\j$};}
  \foreach \i/\j/\g in {
    -1/-1/\FF_2,
    0/-1/\ZZ,
    0/1/\ZZ^2
  }{
    \node at (\i,\j) {$\g$};
    }
\end{tikzpicture}
\]
and therefore the Khovanov homology  over $\FF_2$ of $K$ in these
quantum gradings is
\[
\begin{tikzpicture}[xscale=1.2,yscale=0.35]
  \foreach \i in {-1.5,-0.5,0.5}{
    \draw (\i,-2)--(\i,2);}
  \foreach \j in {-2,0,2}{
    \draw (-1.5,\j)--(0.5,\j);}
  \foreach \i in {-1,0}{
    \node[anchor=south] at (\i,2) {$\i$};}
  \foreach \j in {-1,1}{
    \node[anchor=east] at (-1.5,\j) {$\j$};}
  \foreach \i/\j/\g [count=\c from 1] in {
    -1/-1/\FF_2,
    0/-1/\FF_2^2,
    0/1/\FF_2^2
  }{
    \node (g\c) at (\i,\j) {$\g$};
  }
  \draw[thick,->] (g2)--(g1);
\end{tikzpicture}
\]
where the arrow indicates a rank one Bockstein homomorphism
associated to the coefficient sequence
$0\to\ZZ/2\to\ZZ/4\to\ZZ/2\to0$.

Let $\alpha\in\Kh_{0,-1}(K;\FF_2)$ be the generator of the kernel
of the Bockstein homomorphism, and let $\beta\in\Kh_{0,1}(K;\FF_2)$
be the image of $\alpha$ under the $X$-action (the basepoint
map). If $\Id^\bullet$ denotes the identity cobordism from $K$ to
itself decorated with a single dot, then this $X$-action is the map
$\Id^\bullet_*$ induced on Khovanov homology by $\Id^\bullet$.

For $i\in\{1,2\}$, consider the cobordism
$F_i={D}_i\circ\Id^\bullet\circ \widehat{D}_i$ from $\varnothing$ to
$\varnothing$. Since $F_i$ is a (knotted) dotted sphere,
by a result of Rasmussen and
Tanaka~\cite{Rasmussen-kh-closed,Tanaka-kh-closed} (or, more
precisely, a trivial extension of it~\cite[Lemma 6.16]{RS-mixed}),
$(F_i)_*=({D}_i)_*\circ\Id_*^\bullet\circ (\widehat{D}_i)_*$ is the
identity map on Khovanov homology
$\FF_2=\Kh(\varnothing;\FF_2)\to\Kh(\varnothing;\FF_2)=\FF_2$. Therefore,
both $(\widehat{D}_1)_*$ and $(\widehat{D}_2)_*$ map the generator of
$\Kh(\varnothing;\FF_2)$ (in bigrading $(0,0)$) to $\alpha$, which is
the unique non-zero element of $\Kh_{0,-1}(K;\FF_2)$ in the
kernel of the Bockstein. Therefore, both
$\Id^\bullet_*\circ(\widehat{D}_1)_*$ and $\Id^\bullet_*\circ(\widehat{D}_2)_*$
map the generator of $\Kh(\varnothing;\FF_2)$ to
$\beta\in\Kh_{0,1}(K;\FF_2)$, and both $(D_1)_*$ and $(D_2)_*$
map $\beta$ to the generator of $\Kh(\varnothing;\FF_2)$.
Choose some
$\gamma\in\Kh_{0,1}(K;\FF_2)$ so that $\{\beta,\gamma\}$ is a
basis of $\Kh_{0,1}(K;\FF_2)$.

\begin{figure}
  \centering
  \begin{tikzpicture}

    \foreach\k [count=\c from 0] in {K,\ol{K},\ol{K}_0,\ol{K}_1}{
      \node at (3.5*\c,0) {$\k$};}

    \foreach\rot/\x [count=\c from 0] in {0/0,180/0,0/1,0/2,0/3}{

      \begin{scope}[yshift=3 cm,xshift=3.5*\x cm,xscale=0.5,yscale=0.5,rotate=\rot]

        \foreach \a in {-3,-2,...,3}{
          \coordinate (m\a) at (\a,0);}

        \coordinate (s) at (0,-0.5);
        \coordinate (sw) at (-135:0.5);

        \foreach\a in {0,1}{
          \foreach\b [count=\bi from 0] in {-2,-3,-4}{
            \coordinate (n\a\bi) at (\a,\b);}}

        \draw[knot] (m-2) to[out=-90,in=-90] (m3);
        \draw[knot] (m-1) to[out=-90,in=-90] (m2);

        \draw[knot] (s) to[out=-90,in=90] (n00) to[out=-90,in=90] (n11);
        \draw[knot] (m1) to[out=-90,in=90] (n10) to[out=-90,in=90] (n01) to[out=-90,in=90] (n12);
        \draw[knot] (n11) to[out=-90,in=90] (n02) to[out=-90,in=-90] (m-3);

        \draw[knot] (n12) to[out=-90,in=-135,looseness=1.4] (sw);

        \ifnum\c<3
        \draw[knot] (s)--(m0);
        \ifnum\c<2
        \fill[white] (-180:0.2) arc (-180:0:0.2) to (-180:0.2);
        \fi
        \draw[resol] (sw)--(m0);
        \fi

        \ifnum\c>1
        \draw[knot] (m-3) to[out=90,in=90] (m3);
        \draw[knot] (m-2) to[out=90,in=90] (m2);
        \draw[knot] (m-1) to[out=90,in=90] (m1);
        \fi

        \ifnum\c=3
        \draw[knot] (sw) to[out=45,in=-90] (-0.3,0) to[out=90,in=90] (0.3,0) to[out=-90,in=90] (s);
        \fi

        \ifnum\c=4
        \draw[knot] (sw) to[out=45,in=90] (s);
        \fi
        
        \fill[black] (m0) circle (0.1);
        
      \end{scope}
    }
    
  \end{tikzpicture}
  \caption{\textbf{The strongly invertible knot $K=9_{46}$, its
      quotient knot $\ol{K}$, and the two induced annular trefoil
      knots $\ol{K}_0$ and $\ol{K}_1$.} Here $K$ has been redrawn
    as an intravergent diagram with the axis coming straight out of the
    page through the marked point; it is also the axis for the two
    annular trefoils.}
  \label{fig:9-46-quotients}
\end{figure}
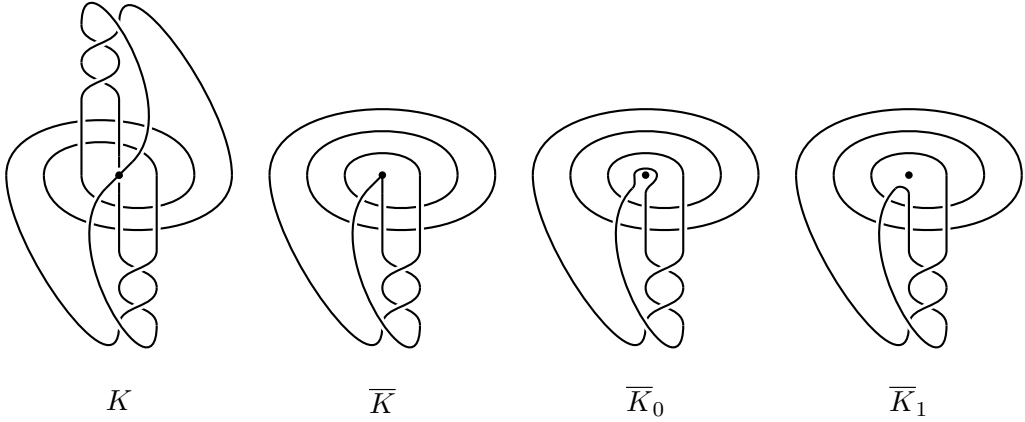

The knot $K$ is strongly invertible with respect to the
$180^\circ$ rotation around the dashed vertical line in
Figure~\ref{fig:9-46-slices}; call the involution $\tau$ and the
induced map on Khovanov homology
$\tau_*\from\Kh(K;\FF_2)\to\Kh(K;\FF_2)$. The two slice
disks $D_1$ and $D_2$ are related by the involution $(\Id,\tau)$ of
$[0,1]\times\RR^3$. Therefore,
$\tau_*\circ\Id^\bullet_*\circ (\widehat{D}_1)_*= \Id^\bullet_*\circ
(\widehat{D}_2)_*$, which implies $\tau_*(\beta)=\beta$.

The annular quotient knots $\ol{K}_0$ and $\ol{K}_1$ are shown in
Figure~\ref{fig:9-46-quotients}. We consider their annular Khovanov
homology in gradings corresponding to the quantum grading $j=1$ on $K$.  (The grading
correction term $\Delta=4$.)
Computer computation, using code by Davis~\cite{Champ-kh-AKh}, gives
that the annular Khovanov homology of $\ol{K}_1$ in gradings with
$2\ol{\jmath}+k+1=12$ is $\FF_2^2$, supported in gradings $(2,7,-3)$
and $(3,5,1)$, while the annular Khovanov homology of $\ol{K}_0$ in
gradings with $2\ol{\jmath}+k=12$ is also $\FF_2^2$, supported in
gradings $(2,7,-2)$ and $(3,5,2)$. It is not hard to find
representatives of these cycles in $\AKh(\ol{K}_1;\FF_2)$ by hand; see
Figure~\ref{fig:comp}. Their images under $f^+$ are (distinct)
nontrivial elements of $\AKh(\ol{K}_0;\FF_2)$, so $f^+$ is an
isomorphism. (Verifying that the image is nontrivial by hand is
straightforward for the cycle in grading $(3,5,1)$, but is quite
tedious for the cycle in grading $(2,7,-3)$, and might be better done
by computer.) Thus,
$\bigoplus_{\ol{\imath},\ol{\jmath},k\mid2\ol{\jmath}+k=1-1+3\Delta}
\AKh_{\ol{\imath},\ol{\jmath},k}(\ol{K}_1,\ol{K}_0;\FF_2)=0$ so, by
Corollary~\ref{cor:comp-tau} the map
$\tau_*\co\Kh_{0,1}(K;\FF_2)\to\Kh_{0,1}(K;\FF_2)$ is given by
$ \left(\begin{smallmatrix}
    0 & 1\\
    1 & 0
  \end{smallmatrix}\right)
  $ with respect to an appropriate basis.

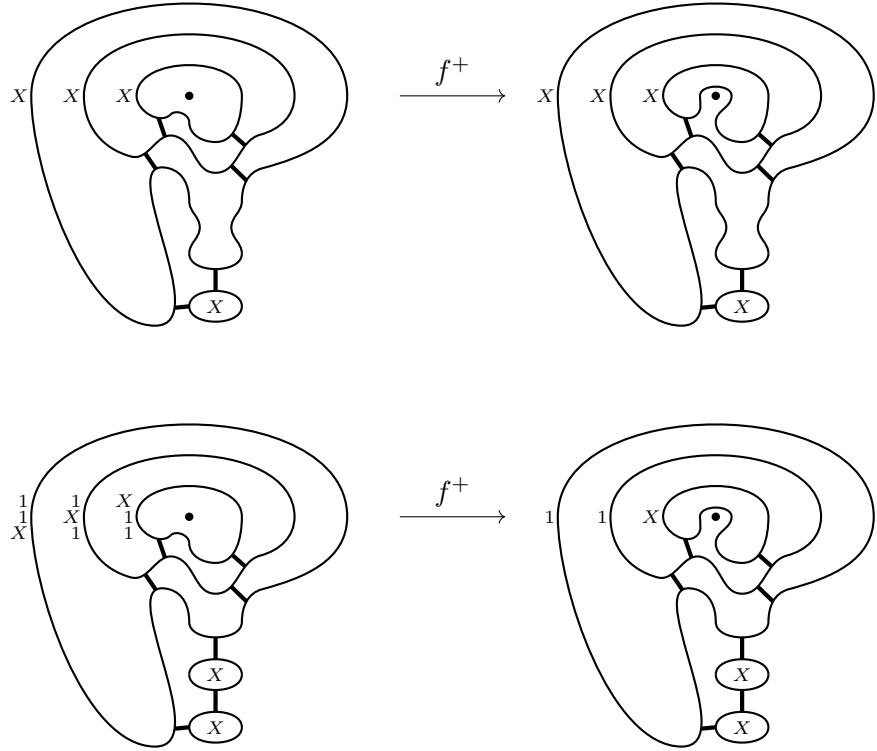
\begin{figure}
  \centering
  \begin{tikzpicture}[scale=0.7]

    \foreach\y in {0,1}{
    \foreach\x in {0,1}{

      \begin{scope}[yshift=8*\y cm,xshift=-10*\x cm,xscale=1,yscale=1]

        \coordinate (o\x) at (0,0);
        
        \foreach \a in {-3,-2,...,3}{
          \coordinate (m\a) at (\a,0);
        }

        \coordinate (s) at (0,-0.5);
        \coordinate (sw) at (-135:0.5);

        \foreach\a in {0,1}{
          \foreach\b [count=\bi from 0] in {-2,-3,-4}{
            \coordinate (n\a\bi) at (\a,\b);
          }}

        \path (m-2) to[out=-90,in=-90] coordinate[pos=0.27] (p0)
        coordinate[pos=0.38] (p1) coordinate[pos=0.5] (p2)
        coordinate[pos=0.62] (p3) (m3) ;

        \path (m-1) to[out=-90,in=-90] coordinate[pos=0.15] (q0)
        coordinate[pos=0.32] (q1) coordinate[pos=0.5] (q2)
        coordinate[pos=0.7] (q3) (m2);
        
        \foreach\a in {0,...,3}{
        }
        
        \path (n02) to[out=-90,in=-90] coordinate[pos=0.3] (n03) (m-3);
        
        \draw[thick] (m-2) to[out=-90,in=165] (p0) to[out=-15,in=165]
        coordinate[pos=0.3] (a3) coordinate[pos=0.8] (a4) (q1)
        to[out=-15,in=180] (p2) to[out=0,in=-165] coordinate[pos=0.3]
        (b5) coordinate[pos=0.7] (b6) (q3) to[out=15,in=-90] (m2)
        to[out=90,in=90] (m-2);

        \draw[thick] (m-3) to[out=-90,in=180,looseness=0.7] (n03) to[out=0,in=165] coordinate[pos=0.2] (b2)
        coordinate[pos=0.9] (b3) (p1) to[out=-15,in=90] (n00) (n10) to[out=90,in=-165] coordinate[pos=0.5] (a5) (p3)
        to[out=15,in=-90] (m3) to[out=90,in=90] (m-3);

        \draw[thick] (n12) to[out=-90,in=-90] (n02)
        to[out=90,in=90] coordinate[pos=0] (a2) coordinate[pos=0.5]
        (a1) (n12) (n11) to[out=-90,in=-90] coordinate[pos=0.5] (b1)
        (n01);

        \draw[thick] (s) to[out=-90,in=180] (q2) to[out=0,in=-90] coordinate[pos=0.35] (a6) (m1)
        to[out=90,in=90] (m-1) to[out=-90,in=-135] coordinate[pos=0.7] (b4) (sw);

        \ifnum\x=0
        \draw[thick] (sw) to[out=45,in=-90] (-0.3,0) to[out=90,in=90] (0.3,0) to[out=-90,in=90] (s);
        \else
        \draw[thick] (sw) to[out=45,in=90] (s);
        \fi

        \ifnum\y=0
        \draw[thick] (n00) to[out=-90,in=-90] coordinate[pos=0.5] (a0) (n10) (n01) to[out=90,in=90] coordinate[pos=0.5] (b0) (n11);
        \draw[ultra thick] (a0) -- (b0);
        \node at ($(n01)!0.5!(n11)$) {\tiny $X$};
        \else
        \draw[thick] (n00) to[out=-90,in=90] ($(n00)!0.5!(n01)+(0.2,0)$) to[out=-90,in=90] (n01) (n10) to[out=-90,in=90] ($(n10)!0.5!(n11)+(-0.2,0)$) to[out=-90,in=90] (n11);
        \fi

        \foreach \a in {1,...,6}{
          \draw[ultra thick] (a\a) -- (b\a);
        }
        
        \fill[black] (m0) circle (0.08);

        \node at ($(n02)!0.5!(n12)$) {\tiny $X$};

        \ifnum\y=1
        \foreach\a in {-3,-2,-1}{
          \node[anchor=east,tightnode] at (m\a) {\tiny $X$};
        }
        \else
        \ifnum\x=0
        \node[anchor=east,tightnode] at (m-3) {\tiny $1$};
        \node[anchor=east,tightnode] at (m-2) {\tiny $1$};        
        \node[anchor=east,tightnode] at (m-1) {\tiny $X$};
        \else
        \foreach \a in {-3,-2,-1}{
          \foreach \b in{-3,-2,-1}{
            \ifnum\a=\b
            \node[anchor=east,tightnode] at ($(m\a)+(0,0.6+0.3*\b)$) {\tiny $X$};
            \else
            \node[anchor=east,tightnode] at ($(m\a)+(0,0.6+0.3*\b)$) {\tiny $1$};
            \fi
          }}
        \fi
        \fi
        
      \end{scope}
    }
    \draw[->] ($(o1)+(4,0)$)--($(o0)+(-4,0)$) node[midway,anchor=south] {$f^+$};
  }
    
  \end{tikzpicture}  
  \caption{\textbf{Cycles in and their images under $f^+$.} Top left:
    a generator for $\AKh_{2,7,-3}(\ol{K}_1;\FF_2)$. Top right: its
    image in $\AKh_{2,7,-2}(\ol{K}_0;\FF_2)$ under $f^+$. Bottom left:
    an element of $\AKh_{3,5,1}(\ol{K}_1;\FF_2)$ which is the sum of
    three terms, where the two inessential circles and one essential
    circle are labeled $X$ and the other two essential circles are
    labeled $1$. (We have labeled the inessential circles and listed
    all possible labels of the three essential circles.) Bottom right:
    its image in $\AKh_{3,5,2}(\ol{K}_0;\FF_2)$. Crossings that were
    $1$-resolved, i.e., where there is a Khovanov differential out of this
    state, are indicated with thick line segments.}
  \label{fig:comp}
\end{figure}

Since $\tau_*(\beta)=\beta$, it follows that the basis must be
$\{\gamma,\gamma+\beta\}$ and
$\tau_*(\gamma)=\gamma+\beta$. So,
\[
  (D_1)_*(\gamma)+(D_2)_*(\gamma)=(D_1)_*(\gamma)+(D_1)_*(\tau_*(\gamma))=(D_1)_*(\beta)=1\in\FF_2=\Kh(\varnothing;\FF_2),
\]
and hence $(D_1)_*(\gamma)\neq (D_2)_*(\gamma)$, as claimed.
\end{proof}

\begin{porism}\label{porism}
  The knot $9_{46}$ does not admit an equivariant slice disk, with
  respect to the $\ZZ/2$-action on $[0,1]\times \RR^3$ by
  $(\Id,\tau)$. In fact, for any slice disk $D$ for $9_{46}$, $D$ and
  $(\Id,\tau)(D)$ are distinguished by the induced maps on Khovanov
  homology.
\end{porism}
\begin{proof}
  The proof is the same as the proof of Theorem~\ref{thm:slice}, with
  $D$ and $(\Id,\tau)(D)$ in place of $D_1$ and $D_2$. 
\end{proof}

Note that the first half of the porism also follows from Sakuma's work~\cite{Sakuma-top-invert}.

\begin{remark}
  For the last portion of the proof of Theorem~\ref{thm:slice}, one
  could instead compute the map $\tau_*$ directly; arguably, that
  involves less work than the argument given. On the other hand, the
  argument above only requires computing the dimensions of certain
  annular Khovanov homology groups, for which there is a well-known,
  fast divide-and-conquer algorithm~\cite{Bar-kh-fast}; computing the
  action of $\tau_*$ efficiently would require further conceptual
  work.
\end{remark}

\section{An analogue in Heegaard Floer homology}\label{sec:HF} 
The main ingredient in the proof of Theorem~\ref{thm:dcov} is a recent
localization theorem for the Heegaard Floer homology of branched
double covers:
\begin{theorem}\cite[Theorem 1.1]{HLL-hf-dcov}\label{thm:HLL}
  Let $Y$ be a closed $3$-manifold, $K\subset Y$ an oriented,
  nullhomologous knot with Seifert surface $F$, and $\spinc$ a
  $\SpinC$-structure on $Y$. Let $\pi\co \Sigma(Y,K)\to Y$ be the
  double cover branched along $K$ induced by the Seifert surface $F$
  and $\pi^*\spinc$ the pullback of $\spinc$ to $\Sigma(Y,K)$.  Then,
  there is a spectral sequence with $E^1$-page given by
  \begin{align*}
    \HFa(\Sigma(Y,K),\pi^*\spinc) &\otimes \FF_2[\theta^{-1}, \theta]]\\
 \shortintertext{converging to}
   \bigoplus_{\{\spinc'\mid \pi^*\spinc' = \pi^*\spinc\}} \HFa(Y,\spinc') &\otimes \FF_2[\theta^{-1}, \theta]].
  \end{align*}
\end{theorem}
\noindent(The cited theorem also discusses links with more components.)

Let
$\pi_{!}\co H^2(\Sigma(Y,K))\to H^2(Y)$ be the transfer
homomorphism. Since $\pi_{!}\circ \pi^*\co H^2(Y)\to H^2(Y)$ is
multiplication by $2$, if $H^2(Y)$ has no 2-torsion then $\pi^*$ is
injective. (In fact, the statement is still true if $H^2(Y)$ has
2-torsion~\cite{Ueki-top-branched}.) So, if $\spinc$ and $\spinc'$ are
distinct $\SpinC$-structures on $Y$ then
$\pi^*\spinc\neq \pi^*\spinc'$ (see~\cite[Lemma 4.7]{HLL-hf-dcov}). In
particular, if $H^2(Y)$ has no 2-torsion, we can sum the spectral
sequence in Theorem~\ref{thm:HLL} over all $\SpinC$-structures to
obtain a spectral sequence
\begin{equation}\label{eq:HLL-sum}
  \HFa(\Sigma(Y,K)) \otimes \FF_2[\theta^{-1}, \theta]]
  \Rightarrow
  \HFa(Y) \otimes \FF_2[\theta^{-1}, \theta]].
\end{equation}

Now, consider a strongly invertible knot $K\subset S^3$ with axis $A$
(a circle). The intersection of $K$ with $A$ decomposes $A$ into two
intervals; label them arbitrarily $A_1$ and $A_2$. Let $\tau\co S^3\to
S^3$ be rotation by $180^\circ$ around $A$ and let $\ol{K}\subset
S^3=S^3/\tau$ be the image of $K\cup A_1$ under the quotient map.

\begin{proof}[Proof of Theorem~\ref{thm:dcov}]
  Let $K'$ be the preimage of $A_2$ in $\Sigma(S^3,\ol{K})$.  We claim
  that $\Sigma(S^3,K)$ is the double cover of $\Sigma(S^3,\ol{K})$
  branched along $K'$. Since $|H^2(\Sigma(S^3,\ol{K}))|=\det(\ol{K})$
  is odd, Formula~\eqref{eq:HLL-sum} then gives the desired spectral
  sequence.

  The claim follows from~\cite[Lemma 3.1]{AB-top-invert-AS}; we
  explain this case of their proof. Let $q\co S^3\to S^3/\tau$ be the
  quotient map. The map $q$ is the double cover branched along
  $q(A)$. Fix a Seifert surface $\ol{F}$ for $\ol{K}$ meeting $q(A_2)$
  transversely, and let $F=q^{-1}(\ol{F})$. Let $Y$ be the result of
  cutting $S^3$ along $F$, so $\bdy Y=F_+\cup_K F_-$. Since $F$ is
  taken to itself by $\tau$, $\tau$ induces an involution of $Y$. The
  fixed set of this involution is a copy of (the preimage of) $A_2$; the two copies of
  $A_1$ (one in $F_+$ and the other in $F_-$) are exchanged by the
  involution.

  We can form $\Sigma(S^3,K)$ by gluing two copies of $Y$
  together. The involution $\tau$ induces an involution $\wt{\tau}$ of
  $\Sigma(S^3,K)$ with fixed set the preimage of $A_2$. The deck
  transformation of $\Sigma(S^3,K)$ gives another involution $\wt{\sigma}$,
  exchanging the two copies of $Y$ and commuting with $\wt{\tau}$. So,
  $\wt{\sigma}$ descends to an involution $\sigma$ of
  $\Sigma(S^3,K)/\wt{\tau}$. The quotient
  $\bigl(\Sigma(S^3,K)/\wt{\tau}\bigr)/\sigma$ is $S^3$, and the fixed
  set of $\sigma$ is the preimage of $\ol{K}$. Thus,
  $\Sigma(S^3,K)/\wt{\tau}=\Sigma(S^3,\ol{K})$ and
  \begin{equation}\label{eq:Kp-cov} 
    \Sigma(S^3,K)=\Sigma(\Sigma(S^3,\ol{K}),K'),
  \end{equation}
  as claimed.
\end{proof}

We conclude with a relative version of Theorem~\ref{thm:dcov}, which
follows from a theorem of Large~\cite{Large-hf-loc} (see
also~\cite{Hen-hf-dcov}). As in the proof of Theorem~\ref{thm:dcov},
the preimage of $A_2$ is a knot $K'$ inside $\Sigma(S^3,\ol{K})$;
let $\wt{K}'$ be the preimage of $K'$ in $\Sigma(S^3,K)$, which is also the
preimage of $A_2$ in $\Sigma(S^3,K)$.
\begin{theorem}\label{thm:HFK-dcov}
  With notation as in Theorem~\ref{thm:dcov}, there is a spectral
  sequence with $E^1$-page given by
  $\HFKa(\Sigma(S^3,K),\wt{K}')\otimes\FF_2[\theta^{-1},\theta]]$
  converging to
  $\HFKa(\Sigma(S^3,\ol{K}),K')\otimes\FF_2[\theta^{-1},\theta]]$.
\end{theorem}
\begin{proof}
  Large proved that given a nullhomologous knot $L$ in a 3-manifold
  $Y$ and a branched double cover $\Sigma(Y,L)$ of $(Y,L)$ there is a
  spectral sequence
  \[
    \HFKa(\Sigma(Y,L),\wt{L})\otimes\FF_2[\theta^{-1},\theta]]\Rightarrow
    \HFKa(Y,L)\otimes\FF_2[\theta^{-1},\theta]],
  \]
  where $\wt{L}$ is the preimage of $L$~\cite[Theorem
  1.5]{Large-hf-loc}. (He states the result as a rank inequality.)  By
  Formula~\eqref{eq:Kp-cov}, Large's theorem with $Y=\Sigma(S^3,\ol{K})$
  and $L=K'$ gives the result.
\end{proof}

\begin{remark}
  The spectral sequence in Theorem~\ref{thm:dcov} decomposes along
  $\SpinC$-structures on $\Sigma(\ol{K})$, and the spectral sequence
  in Theorem~\ref{thm:HFK-dcov} decomposes along $\SpinC$-structures
  on $\Sigma(\ol{K})$ and Alexander gradings, i.e., along relative
  $\SpinC$-structures on $(\Sigma(\ol{K}),K')$.
\end{remark}

\vspace{-0.3cm}
\bibliographystyle{myalpha}
\bibliography{newbibfile}
\vspace{1cm}
\end{document}